\documentclass[11pt]{amsart}
\usepackage{amsmath, amssymb, amscd, mathrsfs, url, pinlabel,verbatim}
\usepackage[pagebackref]{hyperref}
\usepackage[margin=1.25in,marginparwidth=1in,centering,letterpaper,dvips]{geometry}
\usepackage{color,dcpic,latexsym,graphicx,epstopdf,comment}
\usepackage[all]{xy}
\usepackage[dvipsnames]{xcolor}
\usepackage{tikz,tikz-cd,pgfplots}
\usepackage{upquote,tabularx,textcomp}
\usepackage[shortlabels]{enumitem}
\usepackage[color=blue!20!white,textsize=tiny]{todonotes}

\title{Homology lens spaces and $\SL(2,\C)$}

\author{Sudipta Ghosh}
\address{Department of Mathematics\\University of Notre Dame}
\email{sghosh7@nd.edu}

\author{Steven Sivek}
\address{Department of Mathematics\\Imperial College London}
\email{s.sivek@imperial.ac.uk}

\author{Raphael Zentner}
\address{Department of Mathematical Sciences\\Durham University}
\email{raphael.zentner@durham.ac.uk}

\makeatletter
\newtheorem*{rep@theorem}{\rep@title}
\newcommand{\newreptheorem}[2]{%
\newenvironment{rep#1}[1]{%
 \def\rep@title{#2 \ref{##1}}%
 \begin{rep@theorem}}%
 {\end{rep@theorem}}}
\makeatother

\newtheorem {theorem}{Theorem}
\newreptheorem{theorem}{Theorem}
\newtheorem {lemma}[theorem]{Lemma}
\newtheorem {proposition}[theorem]{Proposition}
\newtheorem {corollary}[theorem]{Corollary}

\numberwithin{equation}{section}
\numberwithin{theorem}{section}

\theoremstyle{definition}

\newtheorem{setup}[theorem]{Setup}

\newtheorem{remark}[theorem]{Remark}
\newtheorem*{remark*}{Remark}

\newtheorem{example}[theorem]{Example}

\newlist{pcases}{enumerate}{1}
\setlist[pcases]{
  label=\bf{Case~\arabic*:}\protect\thiscase.~,
  ref=\arabic*,
  align=left,
  labelsep=0pt,
  leftmargin=0pt,
  labelwidth=0pt,
  parsep=0pt
}
\newcommand{\case}[1][]{%
  \if\relax\detokenize{#1}\relax
    \def\thiscase{}%
  \else
    \def\thiscase{~#1}%
  \fi
  \item
}

\newcommand{\Z}{\mathbb{Z}}

\newcommand{\R}{\mathbb{R}}
\newcommand{\C}{\mathbb{C}}

\newcommand{\F}{\mathbb{F}}
\newcommand{\Q}{\mathbb{Q}}

\newcommand{\Hom}{\operatorname{Hom}}

\newcommand\cC{\mathcal{C}}

\newcommand{\RP}{\mathbb{RP}}

\newcommand{\irr}{\text{irr}}

\newcommand\SU{\mathrm{SU}}
\newcommand\SO{\mathrm{SO}}
\newcommand\SL{\mathrm{SL}}

\newcommand\SHI{\mathit{SHI}}

\newcommand\KHI{\mathit{KHI}}

\DeclareFontFamily{U}{mathx}{\hyphenchar\font45}
\DeclareFontShape{U}{mathx}{m}{n}{
      <5> <6> <7> <8> <9> <10>
      <10.95> <12> <14.4> <17.28> <20.74> <24.88>
      mathx10
      }{}
\DeclareSymbolFont{mathx}{U}{mathx}{m}{n}
\DeclareFontSubstitution{U}{mathx}{m}{n}
\DeclareMathAccent{\widecheck}{0}{mathx}{"71}

\DeclareSymbolFont{yhlargesymbols}{OMX}{yhex}{m}{n}
\DeclareMathAccent{\yhwidehat}{\mathord}{yhlargesymbols}{"62}

\newcommand{\inr}{\operatorname{int}}

\newcommand{\ad}{\operatorname{ad}}

\newcommand{\coker}{\operatorname{coker}}
\newcommand{\pt}{\mathrm{pt}}

\makeatletter
\DeclareFontFamily{OMX}{MnSymbolE}{}
\DeclareSymbolFont{MnLargeSymbols}{OMX}{MnSymbolE}{m}{n}
\SetSymbolFont{MnLargeSymbols}{bold}{OMX}{MnSymbolE}{b}{n}
\DeclareFontShape{OMX}{MnSymbolE}{m}{n}{
    <-6>  MnSymbolE5
   <6-7>  MnSymbolE6
   <7-8>  MnSymbolE7
   <8-9>  MnSymbolE8
   <9-10> MnSymbolE9
  <10-12> MnSymbolE10
  <12->   MnSymbolE12
}{}
\DeclareFontShape{OMX}{MnSymbolE}{b}{n}{
    <-6>  MnSymbolE-Bold5
   <6-7>  MnSymbolE-Bold6
   <7-8>  MnSymbolE-Bold7
   <8-9>  MnSymbolE-Bold8
   <9-10> MnSymbolE-Bold9
  <10-12> MnSymbolE-Bold10
  <12->   MnSymbolE-Bold12
}{}

\let\llangle\@undefined
\let\rrangle\@undefined
\DeclareMathDelimiter{\llangle}{\mathopen}%
                     {MnLargeSymbols}{'164}{MnLargeSymbols}{'164}
\DeclareMathDelimiter{\rrangle}{\mathclose}%
                     {MnLargeSymbols}{'171}{MnLargeSymbols}{'171}
\makeatother

\newcounter{desccount}

\newcommand{\descref}[1]{\hyperref[#1]{#1}}

\usetikzlibrary{decorations.markings,arrows,calc,intersections,patterns}
\tikzset{every picture/.style=thick}
\tikzset{link/.style = { white, double = black, line width = 1.75pt, double distance = 1.25pt, looseness=1.75 }}
\tikzset{crossing/.style = {draw, circle, dotted, minimum size=0.5cm, inner sep=0, outer sep=0}}
\pgfplotsset{compat=1.12}
\begin{document}

\begin{abstract}
We prove that if $Y$ is a closed, oriented 3-manifold with first homology $H_1(Y;\Z)$ of order less than $5$, then there is an irreducible representation $\pi_1(Y) \to \SL(2,\C)$ unless $Y$ is homeomorphic to $S^3$, a lens space, or $\RP^3 \# \RP^3$.  By previous work it suffices to consider the case $H_1(Y;\Z) \cong \Z/4\Z$, which we accomplish using holonomy perturbation techniques in instanton Floer homology.
\end{abstract}

\maketitle

\section{Introduction} \label{sec:intro}

If we fix an integer $n \geq 4$, then a classical construction demonstrates that any finitely presented group $G$ can be realized as the fundamental group of some closed $n$-manifold.  The analogous statement is false in three dimensions, leading one to ask which $G$ can actually be realized as 3-manifold groups.  Techniques from both gauge theory and hyperbolic geometry suggest that representation varieties could be a good source of restrictions on $G$, and indeed the third author used these to prove the following:

\begin{theorem}[\cite{zentner}] \label{thm:zentner}
Let $Y$ be an integer homology 3-sphere.  If $Y$ is not homeomorphic to $S^3$, then there is an irreducible representation $\pi_1(Y) \to \SL(2,\C)$.
\end{theorem}

In recent work we generalized this to the cases where $H_1(Y)$ is 2-torsion or 3-torsion.

\begin{theorem}[{\cite[Theorem~1.2]{gsz}}] \label{thm:gsz-reps2}
Let $Y$ be a closed, oriented, connected 3-manifold with $H_1(Y;\Z) \cong (\Z/2\Z)^{\oplus r}$.  If $Y$ is not homeomorphic to $\#^r \RP^3$, then there is an irreducible representation $\pi_1(Y) \to \SL(2,\C)$.
\end{theorem}

\begin{theorem}[{\cite[Theorem~1.3]{gsz}}] \label{thm:gsz-reps3}
Let $Y$ be a closed, oriented, connected 3-manifold with $H_1(Y;\Z) \cong (\Z/3\Z)^{\oplus r}$ for some $r \geq 1$.  If $Y$ is not homeomorphic to $\pm L(3,1)$, then there is an irreducible representation $\pi_1(Y) \to \SL(2,\C)$.
\end{theorem}

The main goal of this paper is to generalize the techniques of \cite{gsz} and add the following additional case.

\begin{theorem} \label{thm:main-Z/4}
Let $Y$ be a closed, orientable 3-manifold with $H_1(Y;\Z) \cong \Z/4\Z$, and suppose that $Y$ is not homeomorphic to a lens space.  Then there is an irreducible representation $\pi_1(Y) \to \SL(2,\C)$.
\end{theorem}

\begin{corollary} \label{cor:homology-small}
Let $Y$ be a closed, orientable 3-manifold with $|H_1(Y;\Z)| < 5$.  Then there is an irreducible representation $\pi_1(Y) \to \SL(2,\C)$ unless $Y$ is $S^3$, a lens space, or $\RP^3\#\RP^3$.
\end{corollary}

\begin{proof}
When $H_1(Y) = 0$, this is the main result of \cite{zentner}.  The cases where $H_1(Y)$ is $\Z/2\Z$ or $(\Z/2\Z)^{\oplus2}$ are part of Theorem~\ref{thm:gsz-reps2}, while Theorem~\ref{thm:gsz-reps3} handles the case $H_1(Y) \cong \Z/3\Z$.  This leaves only $\Z/4\Z$, which is Theorem~\ref{thm:main-Z/4}.
\end{proof}

\subsection{Strategy}

The proof of Theorem~\ref{thm:main-Z/4} follows the same broad outline as the results of \cite{zentner, gsz}.  We say that $Y$ is \emph{$\SL(2,\C)$-reducible} if every representation $\pi_1(Y) \to \SL(2,\C)$ has reducible image, and then our goal is to classify the $\SL(2,\C)$-reducible 3-manifolds $Y$ with $H_1(Y;\Z) \cong \Z/4\Z$.  By appealing to \cite{zentner} we see that any such $Y$ must be irreducible, and then the cases where $Y$ is hyperbolic or Seifert fibered are readily dispatched by appeals to \cite{cs-splittings} and \cite{sz-menagerie} respectively.  Thus by geometrization it suffices to assume that $Y$ has an incompressible torus.

\subsubsection*{Simplification of toroidal counterexamples}

We now write our $\SL(2,\C)$-reducible, toroidal manifold as $Y = M_1 \cup_{T^2} M_2$, where $M_1$ and $M_2$ are compact, irreducible 3-manifolds with incompressible torus boundaries.  If there is a degree-1 map of the form $f: M_1 \to M'_1$ that restricts to a homeomorphism $\partial M_1 \xrightarrow{\cong} \partial M'_1$, then we can extend this to a degree-1 map of the form 
\[ Y = M_1 \cup_{T^2} M_2 \to M'_1 \cup_{T^2} M_2 = Y', \]
where the extension to $M_2$ is by the identity.  (To be precise, we can use the gluing map $\partial M_1 \cong \partial M_2$ coming from $Y$ and pick an identification $\partial M'_1 \cong \partial M_2$ so that the homeomorphism $f|_{\partial M_1}: \partial M_1 \to \partial M_1'$, read through these identifications, becomes the identity on $\partial M_2$, and this obviously extends by the identity.)  As a degree-1 map, this induces a surjection $\pi_1(Y) \to \pi_1(Y')$, implying that $Y'$ must also be $\SL(2,\C)$-reducible.  The same holds for degree-1 maps $M_2 \to M'_2$, so we may simplify $Y$ by applying as many such maps as we like.

In favorable situations there is a natural source of degree-1 maps.  Each of the $M_j$ comes equipped with a \emph{rational longitude}, uniquely defined up to sign, which is a primitive class 
\[ \lambda_j \in H_1(\partial M_j; \Z) \]
whose image in rational homology generates the kernel of the rank-1 map $H_1(\partial M_j; \Q) \to H_1(M_j;\Q)$.  If we assume that $H_1(Y) \cong \Z/4\Z$, then it turns out that either
\begin{enumerate}
\item $\lambda_j$ is nullhomologous, and there is a classical degree-1 ``pinching'' map $M_j \to S^1\times D^2$ (see \cite[Proposition~4.2]{gsz} for details);
\item or $\lambda_j$ has order $2$ in $H_1(M;\Z)$, and then we constructed an analogous pinching map from $M_j$ onto the twisted $I$-bundle over the Klein bottle \cite[Proposition~4.5]{gsz}.
\end{enumerate}
We manage to eliminate the second possibility, by using an explicit understanding of the $\SU(2)$ character variety of the twisted $I$-bundle over the Klein bottle to construct non-abelian representations $\pi_1(Y') \to \SU(2)$.  Thus for the main theorem it ultimately suffices to assume that both $M_1$ and $M_2$ are homology solid tori, and that the Dehn fillings $M_1(\lambda_2)$ and $M_2(\lambda_1)$ are both lens spaces of order $4$.

\subsubsection*{Curves of characters in the pillowcase}

Once we have simplified each $M_j$ as above, we can use techniques from gauge theory to construct a non-abelian representation $\pi_1(Y) \to \SU(2)$.  It suffices to find a pair of representations $\rho_j: \pi_1(M_j) \to \SU(2)$ that agree when restricted to $\pi_1(T^2)$, so we study the images of the $\SU(2)$ character varieties $X(M_j)$ inside the \emph{pillowcase} $X(T^2)$, the $\SU(2)$ character variety of $T^2$.  If these images intersect at a point where at least one of the corresponding $\rho_j$ is non-abelian, then these $\rho_j$ will glue to give the desired non-abelian $\rho: \pi_1(Y) \to \SU(2)$.

Having arranged that each $M_j$ is a homology solid torus with a lens space filling, it follows that the pillowcase image of $M_j$ must
\begin{enumerate}
\item avoid certain line segments in the pillowcase, where non-abelian representations $\pi_1(M_j) \to \SU(2)$ would otherwise descend to non-abelian representations of the corresponding lens space; \label{i:avoid-lines}
\item and contain an essential closed curve in the twice-punctured pillowcase (which is homeomorphic to an open cylinder). \label{i:essential-curve}
\end{enumerate}
Claim~\eqref{i:essential-curve} is Proposition~\ref{prop:cyclic-curve}; to prove it, we argue that the longitudinal filling $M_j(\lambda_j)$ must be irreducible and thus have non-vanishing instanton homology, and then the existence of the essential curve follows from the techniques of \cite{zentner} and \cite{lpcz}.

We now have a pair of simple closed curves $\gamma_1$ and $\gamma_2$ in the pillowcase, corresponding respectively to $\SU(2)$ representations of $\pi_1(M_1)$ and of $\pi_1(M_2)$, and we wish to show that they intersect.  In fact, there is a slight caveat: there is exactly one point $p \in X(T^2)$, denoted $(\frac{\pi}{2},0)$ in our preferred coordinates, that might correspond to abelian representations of both of the $\pi_1(M_j)$.  But if the $\gamma_i$ meet at $p$ then they must do so transversely, and then they must also meet somewhere else, because simple closed curves on $X(T^2) \cong S^2$ cannot have a single transverse point of intersection.  So it suffices to show that $\gamma_1$ intersects $\gamma_2$.

From here the most novel part of the argument, in comparison with our work \cite{gsz}, is the case where exactly one of the $\gamma_j$ (say, $\gamma_2$) passes through the distinguished point $p$.  Writing each $M_j$ as the complement $Y_j \setminus N(K_j)$ of a knot in a homology sphere, we used a non-vanishing theorem of the form
\[ I_*^{w_j}\big( (Y_j)_0(K_j) \big) \neq 0, \]
in conjunction with the holonomy perturbation techniques of \cite{zentner}, to deduce the existence of the curves $\gamma_j$.  Here, by contrast, we apply these techniques to a lower bound
\[ \dim \KHI(Y_j,K_j) > 1 \]
on the \emph{instanton knot homology} of \cite{km-excision}.  Specifically, we use $\gamma_2$, together with the line avoided by $X(M_1)$ as in claim~\eqref{i:avoid-lines} above, to construct a simple closed curve $c$ in the pillowcase with the properties that
\begin{enumerate}
\item if the irreducible characters of $M_1$ avoid the curve $\gamma_2$ in the pillowcase, then they also avoid $c$;
\item and $c$ divides the pillowcase into two regions of equal area.
\end{enumerate}
Following \cite[\S4]{sz-pillowcase}, the holonomy perturbation argument says that the representations $\pi_1(M_1) \to \SU(2)$ along some $C^1$-small approximation of $c$ generate a chain complex for $\KHI(Y_1,K_1)$, and only one of these generators corresponds to an abelian representation, so there must be a non-abelian representation $\rho_1: \pi_1(M_1) \to \SU(2)$ whose image meets $\gamma_2$ after all.  This gives rise to the desired representation $\rho$ of $\pi_1(Y)$, completing the proof.

\subsection{Organization}

We begin with a brief review of the pillowcase in \S\ref{sec:pillowcase}.  In \S\ref{sec:toroidal-decomposition} we reduce Theorem~\ref{thm:main-Z/4} to the case of toroidal manifolds $Y$ where one side of the incompressible torus has a lens space filling, and either the other side is either a homology solid torus with a lens space filling or its rational longitude has reasonably small order.  Then in \S\ref{sec:essential-case} we rule out the second possibility by reducing it to the case where that side is the twisted $I$-bundle over the Klein bottle.  Finally, most of the paper is concentrated in \S\ref{sec:nullhomologous}, where we use instanton gauge theory to associate curves of characters in the pillowcase to each piece of $Y \setminus N(T^2)$ and thus prove that the images of their character varieties must intersect.

\section{The pillowcase} \label{sec:pillowcase}

In this section we briefly review some facts about the pillowcase orbifold, as discussed for example in \cite[\S3.1]{lpcz} or \cite[\S3.1]{gsz}.

Given a manifold $M$, we can define the $\SU(2)$ representation and character varieties of $M$ by
\begin{align*}
R(M) &= \Hom(\pi_1(M), \SU(2)), \\
X(M) &= R(M) / \SO(3),
\end{align*}
where $\SO(3) = \SU(2) / \{\pm1\}$ acts on $R(M)$ by conjugation.  We write $R^\irr(M)$ or $X^\irr(M)$ for the subsets of each variety consisting of irreducible representations.  If $M$ is the exterior of a knot $K \subset Y$ then we will also write $R(Y,K)$ to mean $R(M)$ and so on.  We will say that $M$ is \emph{$\SU(2)$-abelian} if $R(M)$ consists entirely of representations with abelian image.

In the case $M = T^2$, we have $\pi_1(T^2) \cong \Z^2$, say with generators $\mu$ and $\lambda$.  A representation $\rho: \pi_1(T^2) \to \SU(2)$ is then determined by a pair of commuting matrices $\rho(\mu)$ and $\rho(\lambda)$, and since these are in $\SU(2)$ they can be simultaneously diagonalized, so that up to conjugacy we have
\begin{align} \label{eq:pillowcase-coords}
\rho(\mu) &= \begin{pmatrix} e^{i\alpha} & 0 \\ 0 & e^{-i\alpha} \end{pmatrix}, &
\rho(\lambda) &= \begin{pmatrix} e^{i\beta} & 0 \\ 0 & e^{-i\beta} \end{pmatrix}
\end{align}
for some $\alpha,\beta \in \R/2\pi\Z$.  This uniquely determines the conjugacy class of $\rho$ up to replacing $(\alpha,\beta)$ with $(-\alpha,-\beta)$, so we have
\[ X(T^2) \cong \frac{(\R/2\pi\Z)\times(\R/2\pi\Z)}{(\alpha,\beta) \sim (2\pi-\alpha,2\pi-\beta)}. \]
This quotient is homeomorphic to a sphere, but has four orbifold points of order $2$ (where $\alpha,\beta \in \pi\Z$), so it is known as the \emph{pillowcase orbifold}.  In practice we will use a fundamental domain for this and write
\[ X(T^2) \cong \frac{[0,\pi]_\alpha \times [0,2\pi]_\beta}{ \left\{ \begin{array}{c} (0,\beta) \sim (0,2\pi-\beta) \\ (\pi,\beta) \sim (\pi,2\pi-\beta) \\ (\alpha,0) \sim (\alpha,2\pi) \end{array} \right\} }. \]

Now if $M$ is a compact 3-manifold with boundary $T^2$, the inclusion
\[ i: T^2 \cong \partial M \hookrightarrow M \]
gives a homomorphism $i_*: \pi_1(T^2) \to \pi_1(M)$ and hence a pullback map
\[ i^*: X(M) \to X(T^2). \]
We refer to the image $i^*X(M)$ or $i^*X^\irr(M)$ as the \emph{pillowcase image} of $M$.

\begin{example} \label{ex:reducible-nullhomologous}
In the case where $M \cong Y \setminus N(K)$ is the complement of a nullhomologous knot $K \subset Y$, we use the meridian $\mu$ and longitude $\lambda$ of $K$ as our preferred generators of $\pi_1(T^2)$, giving rise to coordinates $(\alpha,\beta)$ on the pillowcase as in \eqref{eq:pillowcase-coords}.  Any abelian representation $\rho: \pi_1(M) \to \SU(2)$ will factor through $H_1(M)$, where $[\lambda]=0$, and thus satisfy $\rho(\lambda) = 1$.  Thus in $i^*X(Y,K)$ the abelian representations will all lie on the line $\beta \equiv 0 \pmod{2\pi}$, and conversely any representation off that line must be non-abelian.
\end{example}

\begin{lemma}[{\cite[Lemma~3.1]{lpcz}}] \label{lem:avoid-edges}
Let $Y$ be an $\SU(2)$-abelian integer homology sphere, and let $K \subset Y$ be a knot.  Then there is some positive $\delta = \delta(K) > 0$ such that the pillowcase image
\[ i^*X^\irr(Y,K) \subset X(T^2) \]
is contained in the region $\delta \leq \alpha \leq \pi-\delta$.
\end{lemma}

In the situation of Lemma~\ref{lem:avoid-edges}, we can lift $i^*X(Y,K)$ to the \emph{cut-open pillowcase}
\[ \cC = [0,\pi]_{\alpha} \times (\R/2\pi\Z), \]
because Example~\ref{ex:reducible-nullhomologous} and Lemma~\ref{lem:avoid-edges} tell us that the image $i^*X(Y,K)$ only meets the lines $\alpha=0$ and $\alpha=\pi$ in the pillowcase at the points $(\alpha,\beta) = (0,0)$ and $(\pi,0)$, both of which have a unique preimage in $\cC$.  Lidman, Pinz\'on-Caicedo, and Zentner \cite{lpcz} proved the following, generalizing a theorem of Zentner \cite[Theorem~7.1]{zentner} in the case $Y=S^3$.

\begin{theorem}[\cite{lpcz}] \label{thm:pillowcase-alternative}
Let $Y$ be an $\SU(2)$-abelian integer homology sphere, and let $K \subset Y$ be a knot with irreducible, boundary-incompressible exterior.  Then the image
\[ i^*X(Y,K) \subset \cC \]
in the cut-open pillowcase contains a topologically embedded curve in the interior of $\cC$ that is homologically nontrivial in $H_1(\cC;\Z) \cong \Z$.
\end{theorem}

\begin{proof}
Since $Y$ is an $\SU(2)$-abelian homology sphere, its instanton homology $I_*(Y)$ is zero.  Then by \cite[Theorem~1.3]{lpcz}, the fact that $Y \setminus N(K)$ is irreducible and boundary-incompressible implies the non-vanishing of $I^w_*(Y_0(K))$, where $w \in H^2(Y_0(K);\Z/2\Z) \cong \Z/2\Z$ is nonzero.  Now we can apply the ``pillowcase alternative'' of \cite[Theorem~3.5]{lpcz} to find the desired curve $\gamma$ in $i^*X(Y,K)$.

The only claim that needs further justification is that $\gamma$ avoids the boundary curves $\{\alpha=0\}$ and $\{\alpha=\pi\}$.  We take the constant $\delta > 0$ provided by Lemma~\ref{lem:avoid-edges} and note that by Example~\ref{ex:reducible-nullhomologous}, any point $(\alpha,\beta) \in \gamma$ with either $0 \leq \alpha < \delta$ or $\pi-\delta < \alpha \leq \pi$ must satisfy $\beta\equiv0\pmod{\pi}$.  If the intersection
\[ I_1 = \gamma \cap \big( [0,\delta/2] \times (\R/2\pi\Z) \big) \subset \cC \]
is nonempty and we let
\[ \alpha_1 = \inf \{ \alpha \in [0,\delta/2] \mid (\alpha,0) \in \gamma \}, \]
then $I_1$ must therefore be precisely the line segment $[\alpha_1,\delta/2] \times \{0\} \subset \cC$: we have $(\alpha_1,0) \in \gamma$ since $\gamma$ is closed, and if any point of $(\alpha_1,\delta/2] \times \{0\}$ were missing from $\gamma$ then $\gamma$ would be disconnected.  By the same argument the intersection
\[ I_2 = \gamma \cap \big( [\pi-\delta/2,\pi] \times (\R/2\pi\Z) \big) \subset \cC \]
must either be empty or have the form $[\pi-\delta/2,\alpha_2] \times \{0\}$.  We let $f_s: [0,\pi] \to [0,\pi]$ be a deformation retraction of $[0,\pi]$ onto $[\delta/2,\pi-\delta/2]$, and then
\[ \cC = [0,\pi] \times (\R/2\pi\Z) \xrightarrow{f_s \times \mathrm{id}} [0,\pi] \times (\R/2\pi\Z) = \cC \]
is a deformation retraction of $\cC$ that restricts to a deformation retraction of $\gamma$.  We replace $\gamma$ with its image $f_1(\gamma) \subset i^*X(Y,K)$, which is the desired curve because it no longer contains any point $(\alpha,\beta)$ with $\alpha < \delta/2$ or $\alpha > \pi - \delta/2$.
\end{proof}

\section{Decompositions of toroidal manifolds} \label{sec:toroidal-decomposition}

In this section, which parallels \cite[\S7]{gsz}, we show that a toroidal manifold $Y$ with homology $H_1(Y) \cong \Z/p^e\Z$ of prime power order must be a union of two relatively simple pieces.  Assuming that such a manifold is $\SL(2,\C)$-reducible, Proposition~\ref{prop:toroidal-simple} will provide a minimal example with respect to the partial ordering given by degree-1 maps.  If we further assume that $H_1(Y) \cong \Z/4\Z$, then we will show in Proposition~\ref{prop:rule-out-4b} that both pieces must have nullhomologous rational longitudes, so that we can focus on this case afterward.

\begin{lemma} \label{lem:toroidal-torsion}
Let $Y$ be a closed, connected, orientable, toroidal 3-manifold with $H_1(Y;\Z) \cong \Z/p^e\Z$ for some prime $p$ and integer $e\geq 0$.  Then we can write
\[ Y = M_1 \cup_{T^2} M_2, \]
where $M_1$ and $M_2$ are compact manifolds with incompressible torus boundary satisfying
\begin{align*}
H_1(M_1;\Z) &\cong \Z, \\
H_1(M_2;\Z) &\cong \Z \oplus \Z/p^f\Z
\end{align*}
for some integer $f$ with $0 \leq f \leq e$.
\end{lemma}

\begin{proof}
A torus in a rational homology sphere must be separating, so we can write $Y = M_1 \cup_{T^2} M_2$ where $M_1$ and $M_2$ are glued along their incompressible torus boundaries.  Since $H_2(Y;\Z) \cong 0$, the Mayer--Vietoris sequence for this decomposition reads
\[ 0 \to \underbrace{H_1(T^2)}_{\cong\Z^2} \xrightarrow{i_*} H_1(M_1) \oplus H_1(M_2) \xrightarrow{j_*} \underbrace{H_1(Y)}_{\cong \Z/p^e\Z} \to 0. \]
Each map $H_1(T^2) \to H_1(M_i)$ induced by inclusion has rank $1$, so $b_1(M_k) \geq 1$ for each $k$, but examining the sequence over $\Q$ shows that $b_1(M_1) + b_1(M_2) = 2$ and so in fact $b_1(M_k) = 1$.  This means that we can write
\[ H_1(M_k;\Z) \cong \Z \oplus T_k \qquad (k=1,2) \]
for some torsion groups $T_k$.

The nonzero elements of $T_1 \oplus T_2$ cannot lie in the image of the injective map $i_*$, since that image is free abelian, so by exactness we have an injection
\[ j_*|_{T_1\oplus T_2}: T_1 \oplus T_2 \hookrightarrow H_1(Y) \cong \Z/p^e\Z. \]
But $\Z/p^e\Z$ cannot be written as a direct sum of two nontrivial groups, so we conclude that one of the $T_k$ must be zero; without loss of generality we label $M_1$ and $M_2$ so that $T_1 = 0$.   Then $T_2$ might be nonzero, but it does inject into $\Z/p^e\Z$, so we must have $T_2 \cong \Z/p^f\Z$ where $0 \leq f \leq e$.
\end{proof}

\begin{lemma} \label{lem:toroidal}
Suppose that $Y$ is $\SL(2,\C)$-reducible and that $H_1(Y)$ is a finite cyclic group, but that $Y$ is not a lens space.  Then $Y$ has an incompressible torus.
\end{lemma}

\begin{proof}
Since $Y$ is $\SL(2,\C)$-reducible it cannot be hyperbolic \cite[Proposition~3.1.1]{cs-splittings}.  An $\SL(2,\C)$-reducible rational homology sphere is also $\SU(2)$-abelian, so if $Y$ is not a lens space then it can only be Seifert fibered with $H_1(Y)$ finite if its base orbifold is either $S^2(2,4,4)$ or $S^2(3,3,3)$ \cite[Theorem~1.2]{sz-menagerie}.  We adapt the proof of \cite[Lemma~8.2]{gsz} to rule these out: if
\[ Y \cong S^2((2,\beta_1),(4,\beta_2),(4,\beta_3)) \]
then there is a surjection 
\[ H_1(Y) \cong \coker\begin{pmatrix} 2 & 0 & 0 & \beta_1 \\ 0 & 4 & 0 & \beta_2 \\ 0 & 0 & 4 & \beta_3 \\ 1 & 1 & 1 & 0 \end{pmatrix} \twoheadrightarrow \Z/2\Z \oplus \Z/4\Z \]
given by sending the respective generators to $(1,2)$, $(1,1)$, $(0,1)$, and $(0,0)$, while if
\[ Y = S^2((3,\beta_1),(3,\beta_2),(3,\beta_3)) \]
then we have a surjection
\[ H_1(Y) \cong \coker\begin{pmatrix} 3 & 0 & 0 & \beta_1 \\ 0 & 3 & 0 & \beta_2 \\ 0 & 0 & 3 & \beta_3 \\ 1 & 1 & 1 & 0 \end{pmatrix} \twoheadrightarrow \Z/3\Z \oplus \Z/3\Z \]
sending the generators to $(1,0)$, $(0,1)$, $(2,2)$, and $(0,0)$.  In either case the first homology can't be cyclic, so this rules out all of the Seifert fibered possibilities, and now we conclude by the geometrization theorem that $Y$ must be toroidal.
\end{proof}

Supposing that $Y = M_1 \cup M_2$ as in Lemma~\ref{lem:toroidal-torsion} is $\SL(2,\C)$-reducible, we wish to arrange for the manifolds $M_1$ and $M_2$ to be as simple as possible.  The following theorem of Rong will help us achieve this goal:

\begin{theorem}[{\cite[Theorem~3.9]{rong}}] \label{thm:rong}
Suppose we have an infinite sequence of closed, oriented $3$-manifolds and degree-1 maps between them, of the form
\[ Y_1 \xrightarrow{f_1} Y_2 \xrightarrow{f_2} Y_3 \xrightarrow{f_3} \cdots. \]
Then the map $f_i$ is a homotopy equivalence for all sufficiently large $i$.
\end{theorem}

We will also use the following lemma.

\begin{lemma} \label{prop:toroidal-irreducible}
Let $M_1$ and $M_2$ be compact 3-manifolds with incompressible torus boundary, and form
\[ Y = M_1 \cup_{T^2} M_2 \]
by gluing them according to some diffeomorphism $\partial M_1 \xrightarrow{\cong} \partial M_2$.  If $Y$ is $\SL(2,\C)$-reducible, and if $H_1(Y;\Z) \cong \Z/n\Z$ where $n \geq 1$ does not have the form $n=2e$ for some odd $e \geq 3$, then each of $M_1$, $M_2$, and $Y$ are irreducible.
\end{lemma}

\begin{proof}
If $Y$ were reducible then it would be a connected sum, say $Y \cong Z \# Z'$, since it does not have the same homology as $S^1\times S^2$.  Both $Z$ and $Z'$ must be $\SL(2,\C)$-reducible if $Z\#Z'$ is, so neither one is a homology 3-sphere by Theorem~\ref{thm:zentner}.  Now if neither $H_1(Z)$ nor $H_1(Z')$ were $2$-torsion, then we could find an irreducible representation
\[ \pi_1(Z \# Z') \cong \pi_1(Z) \ast \pi_1(Z') \twoheadrightarrow H_1(Z) \ast H_1(Z') \to \SL(2,\C), \]
exactly as in \cite[Theorem~1.5]{gsz}.  Thus without loss of generality $H_1(Z)$ is $2$-torsion, and
\[ H_1(Y) \cong H_1(Z) \# H_1(Z') \cong (\Z/2\Z)^{\oplus r} \oplus H_1(Z') \]
for some $r \geq 1$.  Since $H_1(Y)$ is cyclic we conclude that $r=1$ and that $H_1(Z')$ is cyclic of odd order $e > 1$.  But then $|H_1(Y)| = 2e$, and we have assumed that this is not the case, so $Y$ must be irreducible after all.

Now that we know $Y$ to be irreducible, we suppose that one of the $M_i$ is reducible, say $M_1$ without loss of generality.  Then we can write $M_1 \cong W \# M'_1$, where $W \not\cong S^3$ is closed and $\partial M'_1 \cong T^2$ is incompressible.  In this case $Y \cong W \# (M'_1 \cup_{T^2} M_2)$ is a nontrivial connected sum, because $M'_1 \cup_{T^2} M_2$ is a closed toroidal manifold and hence not $S^3$.  This contradicts the irreducibility of $Y$, so $M_1$ must have been irreducible after all.
\end{proof}

\begin{proposition} \label{prop:toroidal-simple}
Suppose that there exists a toroidal, $\SL(2,\C)$-reducible 3-manifold with $H_1(Y;\Z) \cong \Z/p^e\Z$, where $p^e$ is a prime power.  Then there is such an irreducible 3-manifold $Y$ with the following additional properties: if we write
\[ Y = M_1 \cup_{T^2} M_2 \]
as in Lemma~\ref{lem:toroidal-torsion}, and if $\lambda_i$ is the rational longitude of $M_i$, then
\begin{enumerate}
\item both $M_1$ and $M_2$ are irreducible and boundary-incompressible,
\item $H_1(M_1)\cong \Z$,
\item $M_2(\lambda_1)$ is a lens space of order $p^e$,
\item and either
\begin{enumerate}
\item $\lambda_2$ is nullhomologous, $M_1(\lambda_2)$ is a lens space of order $p^e$, $H_1(M_2) \cong \Z$, and $\Delta(\lambda_1,\lambda_2) = p^e$; \label{i:zhs3}
\item or $\lambda_2$ has order $p^k$ for some $k$ with $1 \leq k < e$, and $\Delta(\lambda_1,\lambda_2)$ divides $p^{e-k}$.
\label{i:lambda-order-pk}
\end{enumerate}
\end{enumerate}
\end{proposition}

\begin{proof}
The irreducibility of $Y$ and of the $M_i$ is Lemma~\ref{prop:toroidal-irreducible}, so we will not discuss it further.

Let $Y_0$ be a toroidal, $\SL(2,\C)$-reducible 3-manifold with $H_1(Y_0;\Z) \cong \Z/p^e\Z$.  Supposing that the claimed $Y$ does not exist, we will inductively build an infinite sequence of 3-manifolds and degree-1 maps between them, of the form
\[ Y_0 \xrightarrow{f_0} Y_1 \xrightarrow{f_1} Y_2 \xrightarrow{f_2} \dots, \]
such that the $f_i$ are not homotopy equivalences, and this will contradict Theorem~\ref{thm:rong}.

Given $Y_i$, which is toroidal with $H_1(Y_i;\Z) \cong \Z/p^e\Z$, we use Lemma~\ref{lem:toroidal-torsion} to write
\[ Y_i  \cong M_1^i \cup_{T^2} M_2^i \]
with $H_1(M_1^i) \cong \Z$. Letting $\lambda_j^i$ be the rational longitude of $M_j^i$ for $j=1,2$, it follows that $\lambda_1^i$ has order 1, so there is a degree-1 map that pinches $M_1$ onto a solid torus while sending $\lambda_1^i$ to a longitude of that solid torus.  (This is a classically known construction, but see \cite[Proposition~4.2]{gsz} for details.)  This map extends to a degree-1 map
\[ p_i: Y_i \to M_2^i(\lambda^i_1), \]
which is not a homotopy equivalence because the incompressibility of the separating torus $T^2$ guarantees that $\lambda_1^i$ is a nontrivial element of the kernel of $(p_1)_*: \pi_1(Y_i) \to \pi_1(Y_{i+1})$.  Then $M_2^i(\lambda^i_1)$ satisfies $H_1(M^i_2(\lambda^i_1)) \cong H_1(Y_i) \cong \Z/p^e\Z$, since we have constructed it from $Y_i$ by replacing the homology solid torus $M_1^i$ with an actual solid torus in a way that preserves longitudes.  If it is not a lens space then Lemma~\ref{lem:toroidal} says that it must be toroidal, so we let $Y_{i+1} = M^i_2(\lambda^i_1)$ and $f_i = p_i: Y_i \to Y_{i+1}$, and we continue to the next iteration. 

\vspace{1em}
If we have reached this point then $M_2^i(\lambda^i_1)$ must be a lens space of order $p^e$, and we now consider the rational longitude $\lambda_2^i$ of $M_2^i$.

Supposing for now that $\lambda_2^i$ has order $1$, then Lemma~\ref{lem:toroidal-torsion} says that $H_1(M_2^i) \cong \Z \oplus \Z/p^f\Z$ for some $f \leq e$, with the peripheral subgroup $H_1(\partial M_2^i)$ generating the $\Z$ summand since $M_2^i$ is the complement of a nullhomologous knot.  If we choose peripheral elements $\mu^i_1$ and $\mu^i_2$ dual to the longitudes $\lambda^i_1$ and $\lambda^i_2$, then the $\mu^i_j$ generate the $\Z$ summands of each $H_1(M_j^i)$.  Dropping the ``$i$'' superscripts for now and writing the gluing map $\partial M_1^i \xrightarrow{\cong} \partial M_2^i$ in the form
\begin{align*}
\mu_1 &\sim \mu_2^a \lambda_2^b \\
\lambda_1 &\sim \mu_2^c \lambda_2^d
\end{align*}
with $ad-bc=\pm1$, we then compute that
\[ \Delta(\lambda_1, \lambda_2) = \Delta(\mu_2^c\lambda_2^d, \lambda_2) = |c|, \]
that $\lambda_2^{\pm1} \sim \mu_1^{-c}\lambda_1^a$, and that
\[ \Z/p^e\Z \cong H_1(Y_i) \cong \frac{\Z\oplus\Z}{\substack{(1,0)\sim(0,a)\\(0,0)\sim(0,c)}} \oplus \Z/p^f\Z \cong \Z/c\Z \oplus \Z/p^f\Z, \]
so $(\Delta(\lambda_1,\lambda_2), p^f) = (|c|, p^f)$ must be either $(1,p^e)$ or $(p^e,1)$.

Now since $\lambda_2$ has order $1$, there is again a degree-1 map pinching $M_2^i$ to a solid torus and sending $\lambda_2$ to a longitude of the solid torus; again this map
\[ q_i: Y_i \to M_1^i(\lambda_2) \]
cannot be a homotopy equivalence.  Then $M^i_1(\lambda_2)$ is $\SL(2,\C)$-reducible, with homology
\[ H_1(M^i_1(\lambda_2)) \cong \frac{H_1(M^i_1)}{\lambda_2} \cong \frac{H_1(M^i_1)}{\mu_1^{-c}\lambda_1^a} \cong \Z/|c|\Z, \]
so we have two cases depending on the value of $|c| = \Delta(\lambda_1,\lambda_2)$:
\begin{enumerate}
\item if $|c|=1$ then $H_1(M_1^i(\lambda_2)) = 0$, so $M_1^i(\lambda_2)$ must be $S^3$ \cite{zentner}; \label{i:decomp-case1}
\item while if $|c|=p^e$ then $H_1(M_1^i(\lambda_2)) \cong \Z/p^e\Z$. \label{i:decomp-case2}
\end{enumerate}
We use \cite{gsz} to rule out the first case as follows: let $Z_1 = M_1^i(\lambda_2) \cong S^3$ and $Z_2 = M_2^i(\lambda_1)$, which is a lens space; the cores $K_1 \subset Z_1$ and $K_2 \subset Z_2$ of these fillings are nullhomologous, since their meridians $\mu_1=\lambda_2$ and $\mu_2=\lambda_1$ are at distance one from their longitudes $\lambda_1$ and $\lambda_2$, and their exteriors $M_1 \cong Z_1 \setminus N(K_1)$ and $M_2 \cong Z_2 \setminus N(K_2)$ are irreducible and boundary-incompressible.  We form $Y_i$ by splicing the exteriors of $K_1$ and $K_2$, gluing the meridian of one to the longitude of the other and vice versa, and so \cite[Proposition~9.8]{gsz} gives us an irreducible representation $\rho: \pi_1(Y_i) \to \SU(2)$, contradicting the assumption that $Y_i$ was $\SL(2,\C)$-reducible.

We must therefore be in case~\eqref{i:decomp-case2} above, meaning that $(|c|,p^f) = (p^e,1)$ and that
\[ H_1(M_2^i) \cong \Z \oplus \Z/p^f\Z \cong \Z. \]
If $M_1^i(\lambda_2)$ is a lens space of order $p^e$, then $Y_i$ is the desired $Y$ and we are done.  Otherwise $\lambda_2$ has order $1$ and $H_1(M_1^i(\lambda_2)) \cong \Z/p^e\Z$, but $M_1^i(\lambda_2)$ is not a lens space, and then it must be toroidal by Lemma~\ref{lem:toroidal}.  We let $Y_{i+1} = M_1^i(\lambda_2)$ and $f_i = q_i: Y_i \to Y_{i+1}$, and we continue to the next iteration.

\vspace{1em}
The only case left to consider is where $M_2^i(\lambda^i_1)$ is a lens space of order $p^e$, but $\lambda_2^i$ is not nullhomologous in $M_2^i$; since it is a torsion element of $H_2(M_2^i) \cong \Z \oplus \Z/p^f\Z$ where $f \leq e$, it must have order $p^k$, where $1 \leq k \leq f \leq e$.

Identifying $H_2(M_2^i) \cong \Z \oplus \Z/p^f\Z$, we can choose a generator of the torsion summand so that the peripheral elements of $M_2^i$ have the form
\begin{align*}
\mu_2 &= (a,b), & \lambda_2 &= (0,p^{f-k}).
\end{align*}
Then as above we have a degree-1 map $Y_i \to M_2^i(\lambda_1)$, which preserves $H_1$ because it replaces the homology solid torus $M_1^i$ with an actual solid torus in a longitude-preserving way.  If $\lambda_1 \sim \mu_2^c \lambda_2^d$ for some integers $c$ and $d$, then
\[ \Z/p^e\Z \cong H_1(M_2^i(\lambda_1)) \cong \frac{\Z \oplus \Z/p^f\Z}{ c(a,b) + d(0,p^{f-k}) = 0}
\cong \coker \begin{pmatrix} ca & cb+dp^{f-k} \\ 0 & p^f \end{pmatrix}, \]
so by comparing orders we have
\[ |ac| p^f = p^e, \]
or equivalently $|ac| = p^{e-f}$.  In particular, if $k = e$ then these are each equal to $f$, which is sandwiched between $k$ and $e$, so we have $|ac| = 1$ and thus
\begin{align*}
\mu_2 &= (\pm1, b), &
\lambda_2 &= (0,1)
\end{align*}
as elements of $\Z \oplus \Z/p^f\Z$.  But then these elements span a 2-dimensional subspace of $H_1(M_2^i; \Z/p\Z) \cong \Z/p\Z \oplus \Z/p\Z$, contradicting the ``half lives half dies'' theorem which says that the inclusion-induced map
\[ H_1(\partial M_2^i; \F) \to H_1(M_2^i; \F) \]
has rank 1 over any field $\F$.  We must therefore have a strict inequality $k < e$, and moreover
\[ \Delta(\lambda_1,\lambda_2) = \Delta(\mu_2^c\lambda_2^d, \lambda_2) = |c| \]
divides $p^{e-f}$ and hence $p^{e-k}$ as claimed.
\end{proof}

\section{Toroidal manifolds with essential rational longitudes} \label{sec:essential-case}

Having proved Proposition~\ref{prop:toroidal-simple} for manifolds with homology $\Z/p^e\Z$ in general, we now specialize to $p^e = 4$ in order to rule out the last case of that proposition.  We accomplish this by simplifying our toroidal 3-manifold via \cite[Proposition~1.9]{gsz}, asserting that a compact 3-manifold with torus boundary and a rational longitude of order $2$ admits a degree-1 map onto the twisted $I$-bundle over the Klein bottle.  We thus begin by describing the $\SU(2)$ representations of the latter manifold in the following lemma, which is a corrected version of the second half of \cite[Proposition~4.4]{gsz}.

\begin{lemma} \label{lem:TIKB-reps}
Let $N$ be the twisted $I$-bundle over the Klein bottle, with peripheral Seifert fiber $\sigma$ and rational longitude $\lambda$.  Then every representation $\rho: \pi_1(N) \to \SU(2)$ is conjugate to one satisfying
\[ (\rho(\sigma),\rho(\lambda)) = (e^{i\alpha}, \pm1) \quad\text{or}\quad (\rho(\sigma),\rho(\lambda)) = (-1,e^{i\beta}) \]
for arbitrary $\alpha,\beta \in \R/2\pi\Z$, and all such pairs are realized by some $\rho$.  The image of $\rho$ is non-abelian if and only if $\rho(\lambda) \neq \pm1$.
\end{lemma}

\begin{proof}
Following \cite[Proposition~4.4]{gsz}, we write 
\[ \pi_1(N) \cong \langle a,b \mid aba^{-1} = b^{-1}\rangle, \]
with $\sigma = a^2$ and $\lambda = b$.  Up to conjugacy any $\rho : \pi_1(N) \to \SU(2)$ satisfies $\rho(\sigma) = e^{i\alpha}$ and $\rho(\lambda) = e^{i\beta}$ for some $\alpha, \beta \in \R / 2\pi\Z$, so we will assume throughout the proof that all representations $\rho$ have this form.

We first suppose that $\rho : \pi_1(N) \to \SU(2)$ has abelian image, and we claim that it must then satisfy $\rho(\lambda) = \pm1$.  Indeed, if $\rho(\lambda) = e^{i\beta}$ were different from $\pm1$, then $\rho(a)$ would have to lie in the unique $U(1)$ subgroup through $\rho(\lambda)$, so in fact we could write $\rho(a) = e^{it}$ for some $t$.  But then the relation $\rho(aba^{-1}) = \rho(b^{-1})$ would become $e^{i\beta} = (e^{i\beta})^{-1}$, so $\rho(\lambda) = e^{i\beta}$ must have been $\pm1$ after all, a contradiction.  At the same time, given any $\alpha$ we can define a representation $\rho: \pi_1(Y) \to \SU(2)$ with abelian image by
\begin{align*}
\rho(a) &= e^{i\alpha/2}, &
\rho(b) &= \pm1,
\end{align*}
and then we have $\rho(\sigma) = \rho(a^2) = e^{i\alpha}$ and $\rho(\lambda) = \rho(b) = \pm 1$, so all pairs $(\rho(\sigma),\rho(\lambda)) = (e^{i\alpha},\pm1)$ are realized by abelian representations.

Now suppose that $\rho : \pi_1(N) \to \SU(2)$ has non-abelian image.  Since the element $\sigma = a^2$ is central in $\pi_1(N)$, its image $\rho(\sigma)$ must commute with the entire image of $\rho$, so then $\rho(\sigma) = \pm1$.  If we had $\rho(\sigma) = +1$ then its square root $\rho(a)$ would have to be $\pm1$, but this is not possible if $\rho$ has non-abelian image, so in fact $\rho(\sigma) = -1$.  Up to conjugacy this means that $\rho(a) = j$, and then the relation
\[ j \rho(b) j^{-1} = \rho(a) \rho(b) \rho(a)^{-1} = \rho(b)^{-1} \]
implies that the $j$-component of $\rho(b)$ is zero, so we can conjugate again by something of the form $\cos(\theta) + \sin(\theta) j$ to maintain the relation $\rho(a) = j$ while recovering our preferred form $\rho(b) = e^{i\beta}$.  Any such choice of $\beta$ gives rise to a representation $\rho$ with $\rho(\sigma) = -1$ and $\rho(\lambda) = \rho(b) = e^{i\beta}$, and it has non-abelian image precisely when $j$ does not commute with $e^{i\beta}$, i.e., when $e^{i\beta} \neq \pm1$.
\end{proof}

\begin{remark}
The abelian representations with $\rho(\sigma) \neq 1$ in Lemma~\ref{lem:TIKB-reps} were mistakenly omitted from \cite[Proposition~4.4]{gsz}, though this omission does not affect its application in \cite{gsz}.  In the proof of that proposition we had claimed that if $\rho(a)=e^{js}$ for some $s\not\in\pi\Z$, then ``the relation $\rho(aba^{-1}) = \rho(b^{-1})$ implies that $\rho(a)=\pm j$ and that $\rho(b)$ has zero $j$-component,'' when in fact another possibility is that $s$ is arbitrary and $\rho(b) = \pm1$.
\end{remark}

We now begin to address case~\eqref{i:lambda-order-pk} of Proposition~\ref{prop:toroidal-simple} for manifolds $Y$ with $H_1(Y) \cong \Z/4\Z$.  The following is a special case, in which we take the simplest possible choice of $M_2$ whose rational longitude has order $2$ and consider a very specific gluing map $\partial M_1 \xrightarrow{\cong} \partial M_2$.

\begin{proposition} \label{prop:glue-to-TIKB}
Let $N$ denote the twisted $I$-bundle over the Klein bottle, with peripheral Seifert fiber $\mu_2$ and rational longitude $\lambda_2$.  Form a closed 3-manifold
\[ Y = M_1 \cup_{T^2} N \]
where
\begin{enumerate}
\item $H_1(Y; \Z)$ is finite cyclic and $H_1(M_1; \Z) \cong \Z$;
\item $M_1$ is irreducible and has incompressible torus boundary, with longitude $\lambda_1$;
\item and the gluing map identifies $\lambda_1 \sim \mu_2^{\pm1} \lambda_2^{\vphantom{\pm1}}$.
\end{enumerate}
Then there is a representation
\[ \rho : \pi_1(Y) \to \SU(2) \]
with non-abelian image.
\end{proposition}

\begin{proof}
We write $\lambda_1 \sim \mu_2^\epsilon \lambda_2^{\vphantom\epsilon}$ for some $\epsilon \in \{\pm1\}$ and choose $\mu_1 \subset \partial M_1$ to be the peripheral curve satisfying $\mu_1 \sim \lambda_2^\epsilon$, so that
\[ 
\left\{\begin{aligned}
\mu_1 &\sim \hphantom{\mu_2^{\epsilon}}\lambda_2^{\epsilon} \\
\lambda_1 &\sim \mu_2^{\epsilon} \lambda_2^{\vphantom{k}}
\end{aligned}\right.
\qquad\Longleftrightarrow\qquad
\left\{\begin{aligned}
\mu_2 &\sim \mu_1^{-1} \lambda_1^{\epsilon} \\
\lambda_2 &\sim \mu_1^\epsilon.
\end{aligned}\right.
\]
Then we have
\[ \Delta(\mu_1,\lambda_1) = \Delta(\lambda_2^\epsilon, \mu_2^\epsilon \lambda_2^{\vphantom{\epsilon}}) = 1, \]
so $\mu_1$ is dual to $\lambda_1$ as curves in $\partial M_1$, and in particular $\mu_1$ must generate $H_1(M_1) \cong \Z$.  This means that
\[ Y_1 = M_1(\mu_1) \]
is an integer homology sphere.  We let $K_1 \subset Y_1$ be the core of this filling, with exterior $M_1 \cong Y_1 \setminus N(K_1)$.

We first suppose that $Y_1$ is not $\SU(2)$-abelian.  Then there is a representation
\[ \rho_{Y_1}: \pi_1(Y_1) \to \SU(2) \]
with non-abelian image, and this gives rise to a non-abelian
\[ \rho_{M_1}: \pi_1(M_1) \twoheadrightarrow \frac{\pi_1(M_1)}{\llangle \mu_1\rrangle} \cong \pi_1(Y_1) \xrightarrow{\rho_{Y_1}} \SU(2) \]
satisfying $\rho_{M_1}(\mu_1) = 1$.  Now by Lemma~\ref{lem:TIKB-reps} we can take an abelian representation
\[ \rho_N: \pi_1(N) \to \SU(2) \]
satisfying $\rho_N(\lambda_2) = 1$ and $\rho_N(\mu_2) = \rho_{M_1}(\lambda_1^\epsilon)$, and since $\rho_{M_1}$ and $\rho_N$ satisfy
\begin{align*}
\rho_N(\mu_2) &= \rho_{M_1}(\mu_1^{-1}\lambda_1^\epsilon) \\
\rho_N(\lambda_2) &= \rho_{M_1}(\mu_1^\epsilon),
\end{align*}
they glue together to give a representation $\rho: \pi_1(Y) \to \SU(2)$.  The image of $\rho$ contains that of $\rho_{M_1}$, so it is non-abelian and in this case we are done.

For the remainder of the proof we can suppose that $Y_1$ is $\SU(2)$-abelian.  Since the exterior $M_1$ of $K_1 \subset Y_1$ is irreducible and boundary-incompressible, Theorem~\ref{thm:pillowcase-alternative} provides a topologically embedded curve $\gamma$ in the cut-open pillowcase image
\[ i^*X(Y_1,K_1) \subset \cC = [0,\pi]_\alpha \times (\R/2\pi\Z)_\beta \]
that is homologically nontrivial in $H_1(\cC;\Z)$, and that satisfies
\[ \gamma \subset (0,\pi) \times (\R/2\pi\Z) = \inr(\cC). \]
Now for each possible $\epsilon \in \{\pm1\}$ and each $\theta \in \R/2\pi\Z$, we define a line segment $L^\epsilon_\theta$ from the $\alpha=0$ component of $\partial \cC$ to the $\alpha=\pi$ component of $\partial \cC$ by setting
\[ L^\epsilon_\theta = \{ (\alpha, \epsilon(\alpha+\theta)) \mid 0 \leq \alpha \leq \pi \}. \]
See Figure~\ref{fig:gamma-meets-lines}.
\begin{figure}
\begin{tikzpicture}[style=thick]
\foreach \i in {0,6} {
\begin{scope}[xshift=\i cm]
  \draw plot[mark=*,mark size = 0.5pt] coordinates {(0,0)(3,0)(3,6)(0,6)} -- cycle; 
  \begin{scope}[decoration={markings,mark=at position 0.55 with {\arrow[scale=1]{>>}}}]
    \draw[postaction={decorate}] (3,6) -- (0,6);
    \draw[postaction={decorate}] (3,0) -- (0,0);
  \end{scope}
  \draw[dotted] (0,3) -- (3,3);
  \draw[thin,|-|] (0,-0.4) node[below] {\small$0$} -- node[midway,inner sep=1pt,fill=white] {$\alpha$} ++(3,0) node[below] {\small$\vphantom{0}\pi$};
  \draw[thin,|-|] (-0.3,0) node[left] {\small$0$} -- node[midway,inner sep=1pt,fill=white] {$\beta$} ++(0,6) node[left] {\small$2\pi$};
  \begin{scope}[color=red,style=very thick]
     \draw plot [smooth] coordinates { (1,0) (2,2) (0.5,4) (1,6) };
     \node [above left, inner sep=2pt] at (0.5,4) {$\gamma$};
  \end{scope}
\end{scope}
}
\begin{scope}[color=blue,style=ultra thick]
  \draw (0,0) -- node[left,pos=0.4] {$L^{+1}_0$} ++(3,3);
  \draw (0,3) -- node[right,pos=0.5] {$L^{+1}_\pi$} ++(3,3);
  \draw (6,3) -- node[below left,pos=0.4, inner sep=1pt] {$L^{-1}_\pi$} ++(3,-3);
  \draw (6,6) -- node[above right,pos=0.5, inner sep=1pt] {$L^{-1}_0$} ++(3,-3);
\end{scope}
\end{tikzpicture}
\caption{Left: the curve $\gamma$ meeting each of the paths $L^{+1}_0$ and $L^{+1}_\pi$ in the cut-open pillowcase.  Right: same, but with $L^{-1}_0$ and $L^{-1}_\pi$.}
\label{fig:gamma-meets-lines}
\end{figure}
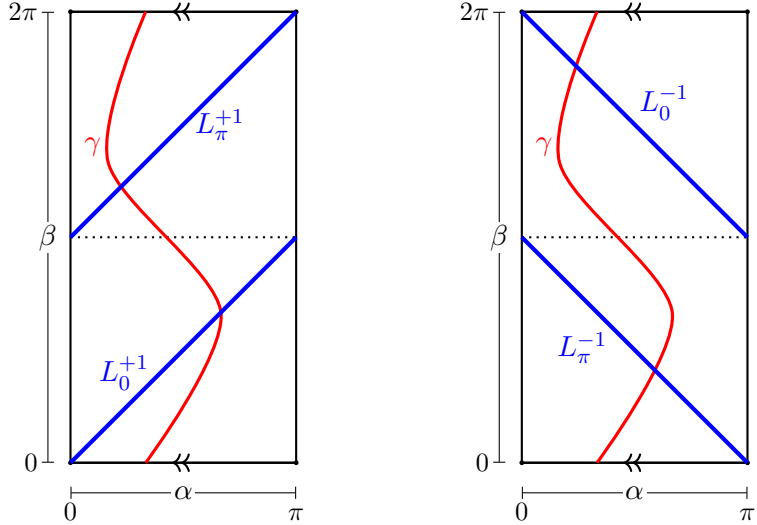
Each of the paths $L^\epsilon_\theta$ generates the relative homology $H_1(\cC, \partial \cC; \Z) \cong \Z$, so the intersection pairing
\[ H_1(\cC) \times H_1(\cC,\partial \cC) \to \Z \]
sends $([\gamma], [L^\epsilon_\theta])$ to $\pm1$, and as such the intersection
\[ \gamma \cap L^\epsilon_\theta \]
must be nonempty.

Any point $(\alpha,\beta)$ of $\gamma \cap L^\epsilon_\theta$ corresponds to a representation
\begin{align*}
\rho_{M_1}: \pi_1(M_1) &\to \SU(2) \\
\mu_1 &\mapsto e^{i\alpha} \\
\lambda_1 &\mapsto e^{i\beta},
\end{align*}
satisfying $\rho_{M_1}( \mu_1^\epsilon ) = e^{i \cdot \epsilon\alpha}$ and
\[ \rho_{M_1}( \mu_1^{-1} \lambda_1^{\epsilon} ) = e^{i(-\alpha+\epsilon\beta)} = e^{i\theta}. \]
We cannot have $\alpha=0$ or $\alpha=\pi$, because $\gamma$ does not contain any such points, and so $\rho_{M_1}(\mu_1^\epsilon) = e^{i\cdot \epsilon\alpha}$ is not equal to $\pm1$.  In other words, for each $\theta\in \R/2\pi\Z$ we have found a representation $\rho_{M_1}: \pi_1(M_1) \to \SU(2)$ satisfying
\begin{align*}
\rho_{M_1}(\mu_1^{-1}\lambda_1^{\epsilon}) &= e^{i\theta}, &
\rho_{M_1}(\mu_1^\epsilon) &= e^{i\cdot \epsilon\alpha} \neq \pm1.
\end{align*}

At this point, we specialize to $\theta = \pi$ and use Lemma~\ref{lem:TIKB-reps} to identify a representation $\rho_N : \pi_1(N) \to \SU(2)$ satisfying
\begin{align*}
\rho_N(\mu_2) &= e^{i\theta} = -1 \\
\rho_N(\lambda_2) &= \rho_{M_1}(\mu_1^\epsilon) \neq \pm1,
\end{align*}
which must necessarily be non-abelian.  Since $\mu_2 \sim \mu_1^{-1}\lambda_1^\epsilon$ and $\lambda_2 \sim \mu_1^\epsilon$, and we have arranged that $\rho_N(\mu_2) = \rho_{M_1}(\mu_1^{-1}\lambda_1^\epsilon)$ and $\rho_N(\lambda_2) = \rho_{M_1}(\mu_1^\epsilon)$, the representations $\rho_{M_1}$ and $\rho_N$ glue together to give us a representation
\[ \rho: \pi_1(Y) \to \SU(2), \]
whose image is non-abelian because it contains the non-abelian image of $\rho_N$.
\end{proof}

\begin{proposition} \label{prop:rule-out-4b}
Assume the hypotheses of Proposition~\ref{prop:toroidal-simple}.  If $p^e = 4$, then case~\eqref{i:lambda-order-pk} of that proposition does not occur.  In other words, case~\eqref{i:zhs3} applies instead: the rational longitude $\lambda_2 \subset \partial M_2$ is nullhomologous, $M_1(\lambda_2)$ is a lens space of order $p^e = 4$, $H_1(M_2) \cong \Z$, and $\Delta(\lambda_1,\lambda_2) = p^e = 4$.
\end{proposition}

\begin{proof}
Supposing that we find ourselves in case~\eqref{i:lambda-order-pk}, the rational longitude $\lambda_2$ must have order exactly $2$, and the distance
\[ \Delta(\lambda_1, \lambda_2) \]
divides $2$.  Since $\lambda_2$ has order $2$, we apply \cite[Proposition~1.9]{gsz} to produce a degree-1 map
\[ M_2 \to N \]
that preserves rational longitudes, where $N$ is the twisted $I$-bundle over the Klein bottle.  This gives rise to a sequence of rational-longitude-preserving, degree-1 maps
\begin{equation} \label{eq:two-pinches}
Y = M_1 \cup_{T^2} M_2 \to M_1 \cup_{T^2} N \to N(\lambda_1),
\end{equation}
where we first pinch $M_2$ to $N$ and then $M_1$ to a solid torus.

It follows from \eqref{eq:two-pinches} that $M_1 \cup_{T^2} N$ and $N(\lambda_1)$ are both $\SL(2,\C)$-reducible, with first homology a quotient of $\Z/4\Z$.  In fact, every Dehn filling of $N$ has homology of order at least $4$ and so we must have
\[ H_1(N(\lambda_1);\Z) \cong \Z/4\Z, \]
which implies in turn that
\[ H_1(M_1 \cup_{T^2} N) \cong \Z/4\Z. \]
We can also read from the proof of \cite[Proposition~4.4]{gsz} that if $\mu_2$ is the Seifert fiber slope on $\partial N$, then
\[ \lambda_1 \sim \mu_2^{\pm1} \lambda_2^{\vphantom{\pm1}} \]
for some choice of sign, since otherwise $N(\lambda_1)$ cannot be $\SL(2,\C)$-reducible with homology $\Z/4\Z$.  But now Proposition~\ref{prop:glue-to-TIKB} provides us with a non-abelian representation
\[ \pi_1(M_1 \cup_{T^2} N) \to \SU(2), \]
so $M_1 \cup_{T^2} N$ cannot be $\SL(2,\C)$-reducible and we have a contradiction.
\end{proof}

\section{Gluing 3-manifolds with nullhomologous rational longitudes} \label{sec:nullhomologous}

With Propositions~\ref{prop:toroidal-simple} and \ref{prop:rule-out-4b} in mind, our main goal in this lengthy section is to prove the following theorem.

\begin{theorem} \label{thm:nullhomologous-case}
Let $Y = M_1 \cup_{T^2} M_2$ be an irreducible 3-manifold with $H_1(Y;\Z) \cong \Z/4\Z$, where
\begin{enumerate}
\item $M_1$ and $M_2$ are irreducible, with incompressible torus boundaries;
\item $H_1(M_1;\Z) \cong H_1(M_2;\Z) \cong \Z$;
\item and if $\lambda_j \subset \partial M_j$ denotes the longitude of $M_j$, then both $M_1(\lambda_2)$ and $M_2(\lambda_1)$ are lens spaces of order $4$.
\end{enumerate}
Then there is a representation $\rho: \pi_1(Y) \to \SU(2)$ with non-abelian image.
\end{theorem}

We will first deduce Theorem~\ref{thm:main-Z/4} from Theorem~\ref{thm:nullhomologous-case} before going on to prove the latter.

\begin{proof}[Proof of Theorem~\ref{thm:main-Z/4}]
Let $Y$ be closed and orientable, with $H_1(Y) \cong \Z/4\Z$, and suppose that $Y$ is not homeomorphic to a lens space.  Then Lemma~\ref{lem:toroidal} says that $Y$ must have an incompressible torus.  Proposition~\ref{prop:toroidal-simple} subsequently produces another such $Y'$ that is both toroidal and $\SL(2,\C)$-reducible, and by Proposition~\ref{prop:rule-out-4b} it satisfies the hypotheses of Theorem~\ref{thm:nullhomologous-case}.  But then Theorem~\ref{thm:nullhomologous-case} gives us a representation
\[ \rho: \pi_1(Y') \to \SU(2) \hookrightarrow \SL(2,\C) \]
with irreducible image, and we have a contradiction.
\end{proof}

We begin in \S\ref{ssec:pillowcase-lens} by gathering some facts about the pillowcase images of knots with lens space surgeries.  In \S\ref{ssec:pillowcase-4-surgery} we specialize to the case where the lens space surgery has slope 4, showing that we can form $Y$ up to an overall orientation reversal by a very specific gluing of knot complements and describing essential curves in the pillowcase images of these knots.  Then in \S\ref{ssec:intersecting-curves} we begin to show that such $Y$ admit non-abelian representations $\rho: \pi_1(Y) \to \SU(2)$, using these essential curves to find the desired $\rho$ when neither curve passes through the point $(\frac{\pi}{2},0)$ in the pillowcase.  In \S\ref{ssec:khi-representations} we apply instanton knot homology to settle the remaining case, where at least one curve contains $(\frac{\pi}{2},0)$, and this completes the proof.

\subsection{Pillowcase images of knots with lens space surgeries} \label{ssec:pillowcase-lens}

Motivated by Proposition~\ref{prop:toroidal-simple}, we prove some general facts about the pillowcase images of knots in homology spheres that admit lens space surgeries.  We begin with the observation that a cyclic surgery forces the pillowcase image to avoid lines of the corresponding slope in the pillowcase.

\begin{lemma} \label{lem:cyclic-surgery-su2}
Let $K$ be a knot in an integer homology sphere $Y$, with complement $E_K$, and with meridian and longitude $\mu,\lambda \subset \partial E_K$.  Suppose that $Y_{r/s}(K)$ is a lens space for some relatively prime $r,s$ with $r\neq 0$ and $s \geq 1$.  Then any representation
\[ \rho: \pi_1(E_K) \to \SU(2) \]
such that $\rho(\mu^r\lambda^s) = \pm1$ must have finite cyclic image and satisfy $\rho(\lambda)=1$.  In other words, the intersection
\[ i^*X^\irr(Y,K) \cap \{ (\alpha,\beta) \mid r\alpha+s\beta \equiv 0\!\!\!\pmod{\pi} \} \subset X(T^2) \]
is empty.
\end{lemma}

\begin{proof}
Let $\rho$ be such a representation.  Then $\ad\rho: \pi_1(E_K) \to \SO(3)$ sends $\mu^r\lambda^s$ to $1$, so it descends to a representation
\[ \pi_1(Y_{r/s}(K)) \cong \frac{\pi_1(E_K)}{\llangle \mu^r\lambda^s \rrangle} \to \SO(3) \]
with the same image, and then $\ad\rho$ must have finite cyclic image because $\pi_1(Y_{r/s}(K))$ is finite cyclic.  The image of $\rho$ is therefore a finite subgroup $G \subset \SU(2)$ whose image under $\ad: \SU(2) \to \SO(3)$ is cyclic, and it follows that $G$ is finite cyclic as well.  Now since $\rho$ has abelian image it factors through $H_1(E_K)$, in which $[\lambda]=0$, and so $\rho(\lambda)=1$ as claimed.
\end{proof}

See Figure~\ref{fig:4-avoiding} for an illustration in the case where $\frac{r}{s}=4$.

We can also show that the image $i^*X^\irr(Y,K)$ avoids an open neighborhood of each point $(\frac{k\pi}{|r|}, 0)$.  The following generalizes \cite[Lemma~10.4]{gsz}, which addressed the case where the slope $\frac{r}{s}$ has numerator an odd prime.

\begin{proposition} \label{prop:zero-root-of-unity}
Let $K$ be a knot in a homology sphere $Y$, with irreducible and boundary-incompressible complement.  Suppose that $Y_{r/s}(K)$ is a lens space, where $r \neq 0$ and $s\geq 1$ are relatively prime.  Then there is an open neighborhood of the set 
\[ \{ r\alpha + s\beta \equiv 0 \!\!\! \pmod{\pi} \} \subset X(T^2), \]
including each point $( \frac{k\pi}{|r|}, 0 )$ with $0 \leq k \leq |r|$, that is disjoint from the image $i^*X^\irr(Y,K)$ of the \emph{irreducible} representations of $\pi_1(E_K)$.  In particular, this neighborhood intersects the pillowcase image $i^*X(Y,K)$ only along the line segment $\beta\equiv0\pmod{2\pi}$.
\end{proposition}

\begin{proof}
If we cannot find such a neighborhood, then some $(\alpha_0,\beta_0)$ with $r\alpha_0+s\beta_0 \in \pi\Z$ is a limit point of $i^*X^\irr(Y,K)$.  There is then a sequence of irreducible representations $\pi_1(E_K) \to \SU(2)$ whose images in the pillowcase converge to $(\alpha_0,\beta_0)$, and since $R(Y,K)$ is compact, some subsequence converges in $R(Y,K)$ to a representation $\rho: \pi_1(E_K) \to \SU(2)$ with pillowcase coordinates $i^*[\rho] = (\alpha_0,\beta_0)$.  By Lemma~\ref{lem:cyclic-surgery-su2} it follows that $\rho$ has abelian image and that $\beta_0=0$, and then $(\alpha_0,\beta_0) = (\frac{k\pi}{|r|},0)$ for some integer $k$ with $0 \leq k \leq |r|$.  In other words, up to conjugacy $\rho$ satisfies
\begin{align*}
\rho(\mu) &= e^{i\cdot k\pi/|r|}, &
\rho(\lambda) &= 1.
\end{align*}

Given that $\rho$ is a reducible limit of irreducible representations, with $\rho(\mu) = e^{i\cdot k\pi/|r|}$, we know once again from \cite[Lemma~3.2]{gsz} that $k$ is neither $0$ nor $|r|$.  Heusener, Porti, and Su\'arez Peir\'o \cite[Theorem~2.7]{heusener-porti-suarez} (cf.\ Klassen \cite{klassen}) moreover proved that the Alexander polynomial of $K \subset Y$ must satisfy
\[ \Delta_K\left( e^{i\cdot 2k\pi/|r|} \right) = 0. \]
Thus by \cite[Theorem~8.21]{burde-zieschang} the branched $|r|$-fold cyclic cover $\tilde{Y} = \Sigma_{|r|}(K)$ of $K\subset Y$ has $b_1(\tilde{Y}) > 0$.  If $\tilde{K} \subset \tilde{Y}$ is the lift of the branch locus $K$, then its meridian $\mu_{\tilde{K}}$ is a lift of $\mu_K^{|r|}$ while the longitude $\lambda_{\tilde{K}}$ lifts $\lambda_K$, and so the $|r|$-fold covering $\tilde{Y} \setminus N(\tilde{K}) \to Y\setminus N(K)$ extends to an $|r|$-fold covering
\[ \tilde{Y}_{\operatorname{sign}(r)/s}(\tilde{K}) \to Y_{r/s}(K). \]
But $Y_{r/s}(K)$ is a lens space of order $|r|$, so we must have $\tilde{Y}_{\pm1/s}(\tilde{K}) \cong S^3$.  In particular $\tilde{Y}$ is actually a homology sphere and we have a contradiction.
\end{proof}

\begin{figure}
\begin{tikzpicture}[style=thick]
\begin{scope}
  \draw plot[mark=*,mark size = 0.5pt] coordinates {(0,0)(3,0)(3,6)(0,6)} -- cycle; 
  \begin{scope}[decoration={markings,mark=at position 0.55 with {\arrow[scale=1]{>}}}]
    \draw[postaction={decorate}] (3,6) -- (0,6);
    \draw[postaction={decorate}] (3,0) -- (0,0);
  \end{scope}
  \begin{scope}[decoration={markings,mark=at position 0.55 with {\arrow[scale=1]{>>}}}]
    \draw[postaction={decorate}] (0,0) -- (0,3);
    \draw[postaction={decorate}] (0,6) -- (0,3);
  \end{scope}
  \begin{scope}[decoration={markings,mark=at position 0.575 with {\arrow[scale=1]{>>>}}}]
    \draw[postaction={decorate}] (3,0) -- (3,3);
    \draw[postaction={decorate}] (3,6) -- (3,3);
  \end{scope}
  \draw[dotted] (0,3) -- (3,3);
  \draw[thin,|-|] (0,-0.4) node[below] {\small$0$} -- node[midway,inner sep=1pt,fill=white] {$\alpha$} ++(3,0) node[below] {\small$\vphantom{0}\pi$};
  \draw[thin,|-|] (0.75,-0.4) node[below] {\small$\frac{\vphantom{3}\pi}{4}$} -- node[midway,inner sep=1pt,fill=white] {$\alpha$} ++(1.5,0) node[below] {\small$\vphantom{0}\frac{3\pi}{4}$};
  \draw[thin,|-|] (-0.3,0) node[left] {\small$0$} -- node[midway,inner sep=1pt,fill=white] {$\beta$} ++(0,6) node[left] {\small$2\pi$};
\end{scope}
\begin{scope}[style=ultra thick]
  \clip (0,0) rectangle (3,6);
  \foreach \i in {1,...,5} {
    \def\lcolor{\ifodd\i{red}\else{purple!50}\fi};
    \draw[\lcolor] (0,3*\i) -- ++(3,-12);
  }
\end{scope}
\node[red,left] at (0.75, 0.5) {$L_\pi$};
\node[purple,right] at (1.875, 4.25) {$L_0$};
\begin{scope}[color=blue,style=ultra thick]
   \draw plot [smooth] coordinates { (1.7,0) (1.6,2) (0.65,4) (0.65,6) } -- (1.7,6);
   \draw (0.65,0) -- (1.7,0);
   \node [left, inner sep=2pt] at (1.6,2) {$\gamma$};
\end{scope}
\end{tikzpicture}
\caption{If $Y_4(K)$ is a lens space, then the pillowcase image $i^*X(Y,K)$ must avoid the lines $L_0 = \{4\alpha+\beta\equiv 0 \pmod{2\pi}\}$ and $L_\pi = \{4\alpha+\beta \equiv \pi \pmod{2\pi} \}$, except at the points where $\beta\equiv0\pmod{2\pi}$.  The curve $\gamma \subset i^*X(Y,K)$ shown here, as provided by Proposition~\ref{prop:cyclic-curve}, only meets these lines at $(\frac{\pi}{4},0)$ and $(\frac{\pi}{2},0)$.}
\label{fig:4-avoiding}
\end{figure}
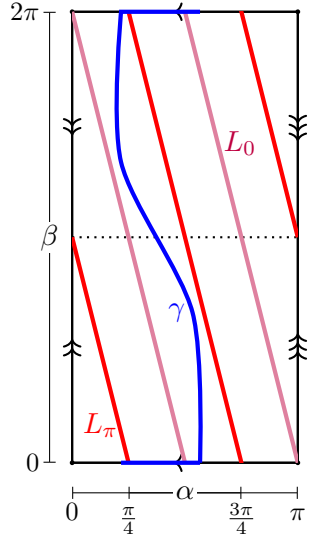

Given a knot $K\subset Y$ with a lens space surgery, we will now show, following \cite{zentner}, that under mild hypotheses the pillowcase image of $K$ must contain an essential simple closed curve in the twice-punctured pillowcase.  This relies on a nonvanishing result for the instanton homology of the zero-surgery $Y_0(K)$, for which by \cite{km-excision} it will suffice to know that $Y_0(K)$ is irreducible.  We thus prove the following generalization of part of \cite[Proposition~6.2]{gsz}, referring the reader to \cite{scott} for basic facts about Seifert fibered spaces.

\begin{proposition} \label{prop:zero-irreducible}
Let $Y$ be a homology 3-sphere, and let $K \subset Y$ be a knot with irreducible, boundary-incompressible exterior.  Suppose that $Y_{r/s}(K)$ is a lens space for some relatively prime integers $r$ and $s$, with $|r| \geq 2$ and $s \geq 1$.  Then $Y_0(K)$ is irreducible.
\end{proposition}

\begin{proof}
Suppose instead that $Y_0(K)$ is reducible.  Then $K$ has a cyclic surgery of slope $\frac{r}{s}$ and a reducible surgery of slope $0$, and the distance between these slopes is $\Delta(\frac{r}{s},\frac{0}{1}) = |r| \geq 2$, so Boyer and Zhang \cite[Theorem~1.2(1)]{boyer-zhang-seminorms} proved that $E_K = Y \setminus N(K)$ is either a simple (i.e., irreducible and atoroidal) Seifert fibered manifold or a cable on the twisted I-bundle over the Klein bottle.  It cannot be the latter, because the Klein bottle would then represent a nonzero class in $H_2(Y;\Z/2\Z) = 0$ -- it does not even separate its own tubular neighborhood, so there is a closed curve dual to it -- so $E_K$ must be simple Seifert fibered.

Next, we claim that $E_K$ has orientable base orbifold.  Indeed, the Seifert fibration of $E_K$ extends over any Dehn filling $E_K(\gamma)$ as long as $\gamma$ is not the Seifert fiber slope, so it extends over all but at most one of the surgeries $Y_{1/n}(K)$, where $n \in \Z$.  If the base orbifold $\Sigma$ of some $Y_{1/n}(K)$ were non-orientable, then we could pull the Seifert fibration back along the orientation double cover $\tilde\Sigma \to \Sigma$ to get a double cover of $Y_{1/n}(K)$, which is impossible because $Y_{1/n}(K)$ is a homology sphere.  Thus $\Sigma$ must be orientable, and the base orbifold of $E_K$ is orientable as well because it is $\Sigma$ minus an open disk.

We now show that $Y_0(K) \cong S^1\times S^2$.  If the longitude $\lambda$ of $K$ is not the Seifert fiber slope on $\partial E_K$, then the Seifert fibration on $E_K$ extends to $E_K(\lambda) \cong Y_0(K)$, and then $Y_0(K)$ is reducible (by assumption) and Seifert fibered but does not have the homology of $\RP^3\#\RP^3$, so it must be $S^1\times S^2$.  Otherwise, if $\lambda$ is the Seifert fiber slope, then since the base orbifold of $E_K$ is orientable we know that $Y_0(K)$ will be a connected sum of several copies of $S^1\times S^2$ and lens spaces \cite[Proposition~2]{heil}, and then the only way that we can have $H_1(Y_0(K)) \cong \Z$ is if $Y_0(K) \cong S^1\times S^2$.

To summarize, we now know that $E_K = Y \setminus N(K)$ is Seifert fibered, with orientable base orbifold; that $Y_{r/s}(K)$ is a lens space; and that $Y_0(K) \cong S^1\times S^2$, so $E_K$ is also the exterior of a knot $K' \subset S^1\times S^2$.  From these facts, Baker, Buck, and Lecuona \cite[Theorem~1.18]{baker-buck-lecuona} proved that $K'$ must be either an $(a,b)$-torus knot or a $(2,\pm1)$-cable of a torus knot, and $K'$ cannot be cabled because we know from the above application of \cite{boyer-zhang-seminorms} that $E_K \cong E_{K'}$ is atoroidal.  As $K'$ is a $(a,b)$-torus knot in $S^1\times S^2$, which in \cite{baker-buck-lecuona} means a cable of $S^1\times\{\pt\}$ with winding number $a$, its complement $E_{K'}$ has first homology $H_1(E_{K'}) \cong \Z \oplus \Z/a\Z$.  But we also know that $H_1(E_K) \cong \Z$, so $|a|=1$ and therefore $K'$ is isotopic to $S^1\times \{\pt\}$.  Then $E_K \cong E_{K'}$ must be a solid torus, and this contradicts the incompressibility of $\partial E_K$, so $Y_0(K)$ must have been irreducible after all.
\end{proof}

\begin{proposition} \label{prop:cyclic-curve}
Let $K$ be a knot in a homology sphere $Y$, with irreducible, boundary-incompressible exterior.  Suppose that $Y_{r/s}(K)$ is a lens space for some relatively prime integers $r,s$ with $|r|\geq 2$ and $s \geq 1$.  Then the pillowcase image $i^*X(Y,K) \subset X(T^2)$ does not contain either of the points $(0,\pi)$ and $(\pi,\pi)$, and moreover there is a topologically embedded, closed curve
\[ \gamma \subset i^*X(Y,K) \]
that is homologically essential in the twice-punctured pillowcase $X(T^2) \setminus \{ (0,\pi), (\pi,\pi) \}$, and that does not contain either $(0,0)$ or $(\pi,0)$.
\end{proposition}

\begin{proof}
Let $E_K$ be the exterior of $K$.  If there were a representation $\rho: \pi_1(E_K) \to \SU(2)$ corresponding to either $(\alpha,\beta)=(0,\pi)$ or $(\alpha,\beta)=(\pi,\pi)$ in the pillowcase, then it would satisfy $\rho(\mu) = \pm1$ and $\rho(\lambda) = -1$, so that $\rho(\mu^r\lambda^s) = \pm1$.  Lemma~\ref{lem:cyclic-surgery-su2} says that any such representation necessarily has $\rho(\lambda)=1$, so the claimed $\rho$ cannot exist.

Now we observe that $Y_0(K)$ is irreducible by Proposition~\ref{prop:zero-irreducible}, and so \cite[Theorem~7.21]{km-excision} says that $I^w_*(Y_0(K)) \neq 0$, where $w$ is the nonzero element of $H^2(Y_0(K);\Z/2\Z) \cong \Z/2\Z$.  Since we have also shown that $i^*X(Y,K)$ avoids $(0,\pi)$ and $(\pi,\pi)$, we can apply \cite[Proposition~3.1]{gsz}, which is really a slight generalization of results of \cite[\S7]{zentner}, to find a homologically essential curve $\gamma \subset i^*X(Y,K)$.

We must now show that $\gamma$ can be arranged to avoid $(0,0)$ and $(\pi,0)$.  To do so, we observe that since $Y_{r/s}(K)$ is a lens space and therefore $\SU(2)$-abelian, we know by \cite[Lemma~3.2]{gsz} that neither $(0,0)$ nor $(\pi,0)$ is a limit point of the image $i^*X^\irr(Y,K)$ of the irreducible character variety of $K$.  This means that each of these points has an open neighborhood in the pillowcase whose intersection with $i^*X(Y,K)$ consists only of the line $\beta \equiv 0 \pmod{2\pi}$, realized by reducible characters.  There is therefore a deformation retraction of $i^*X(Y,K)$ taking it into the complement of these neighborhoods, just as in the proof of Theorem~\ref{thm:pillowcase-alternative}, and the image of $\gamma$ under this deformation retraction will be the desired curve, since it is still contained in $i^*X(Y,K)$ but does not pass through either $(0,0)$ or $(\pi,0)$.
\end{proof}

\subsection{Knots with cyclic surgeries of order 4} \label{ssec:pillowcase-4-surgery}

We now apply the results of the preceding subsection to the case where the prime power in question is $4$.  We must first see how to decompose a toroidal manifold with homology $\Z/4\Z$ into a pair of knot exteriors glued along their boundaries.

\begin{proposition} \label{prop:order-4-gluing}
Suppose that $Y = M_1 \cup_{T^2} M_2$, where $M_1$ and $M_2$ are compact oriented 3-manifolds with torus boundary satisfying
\begin{align*}
H_1(Y) &\cong \Z/4\Z, &
H_1(M_1) \cong H_1(M_2) &\cong \Z.
\end{align*}
Then up to possibly reversing the orientation of $Y$, we can write $M_i = Y_i \setminus N(K_i)$, where each $Y_i$ is a homology sphere, the knot $K_i \subset Y_i$ has meridian $\mu_i$ and longitude $\lambda_i$, and the gluing map identifies
\begin{equation} \label{eq:4-gluing}
\begin{aligned}
\mu_1 &\sim \mu_2, \\
\lambda_1^{-1} &\sim \mu_2^4\lambda_2.
\end{aligned}
\end{equation}
In particular, we have $(Y_1)_4(K_1) \cong M_1(\lambda_2)$ and $(Y_2)_4(K_2) \cong M_2(\lambda_1)$.
\end{proposition}

\begin{proof}
We choose peripheral curves $\mu_i \subset \partial M_i$ that are dual to the longitudes $\lambda_i$, where if $\lambda_1$ is dual to $\lambda_2$ then we will take $\mu_1 = \lambda_2$ and $\mu_2 = \lambda_1$.  Then we let
\[ Y_i = M_i(\mu_i) \qquad (i=1,2), \]
so that each $Y_i$ is a homology sphere, and the core of each Dehn filling is a nullhomologous knot $K_i \subset Y_i$ with complement $M_i$.  If we write the gluing map $\partial M_1 \xrightarrow{\cong} \partial M_2$ as
\begin{equation} \label{eq:zhs3-gluing}
\begin{aligned}
\mu_1 &\sim \mu_2^a\lambda_2^b \\
\lambda_1 &\sim \mu_2^c\lambda_2^d,
\end{aligned}
\end{equation}
with $ad-bc=-1$, then $(Y_2)_{c/d}(K_2) \cong M_2(\lambda_1)$ has homology $\Z/4\Z$, so $|c|=4$.

We can replace $Y_1$ with some $\frac{1}{k}$-surgery on $K_1$, and $K_1$ with the core of this surgery; this preserves the longitude $\lambda_1$ and the complement $M_1$ but replaces $\mu_1$ with $\mu_1\lambda_1^k$, so in these new coordinates the gluing map is
\begin{align*}
\mu_1 &\sim \mu_2^{a+kc}\lambda_2^{b+kd} \\
\lambda_1 &\sim \mu_2^c\lambda_2^d.
\end{align*}
We know that $a$ must be odd since $ad-bc=-1$ and $c=\pm4$, so $a\equiv \pm1\pmod{4}$ and this means that we can choose $k$ so that $a+kc=\pm 1$.  In other words, we can choose $\mu_1$ so that $a = \pm1$ in \eqref{eq:zhs3-gluing}.  Having done this, we reverse the orientation of $K_2$ if needed to fix $a=1$, so that we have an identification of the form $\mu_1 \sim \mu_2 \lambda_2^b$; and then we replace $Y_2$ and $K_2$ with $\frac{1}{b}$-surgery on $K_2$ and the core of this surgery.

At this point we have arranged for the gluing map \eqref{eq:zhs3-gluing} to have the form
\begin{align*}
\mu_1 &\sim \mu_2 \\
\lambda_1 &\sim \mu_2^c\lambda_2^d
\end{align*}
for some integers $c=\pm 4$ and $d$, and the condition $ad-bc=-1$ means that $d=-1$.  If $c=4$ then we can reverse the orientations of both $Y_1$ and $Y_2$ (and hence of $Y$ itself), while fixing the orientations of the $K_i$; this fixes $\mu_1$ and $\mu_2$ while replacing $\lambda_1$ and $\lambda_2$ with their inverses, so that $\lambda_1 \sim \mu_2^c \lambda_2^{-1}$ becomes $\lambda_1 \sim \mu_2^{-c} \lambda_2^{-1}$, and thus we are left with $c=-4$ instead.  Now we have $\mu_1 \sim \mu_2$ and $\lambda_1 \sim \mu_2^{-4}\lambda_2^{-1}$ as in \eqref{eq:4-gluing}, and moreover
\begin{align*}
M_1(\lambda_2^{-1}) &\cong M_1(\mu_1^4\lambda_1) \cong (Y_1)_4(K_1), \\
M_2(\lambda_1^{-1}) &\cong M_2(\mu_2^4\lambda_2) \cong (Y_2)_4(K_2),
\end{align*}
exactly as claimed.
\end{proof}

We introduce a pair of involutions of the pillowcase $X(T^2)$, given in $(\alpha,\beta)$ coordinates by
\begin{align}
\sigma(\alpha,\beta) &= (\alpha, 2\pi-(4\alpha+\beta)), \label{eq:sigma} \\
\tau(\alpha,\beta) &= (\pi-\alpha,2\pi-\beta). \label{eq:tau}
\end{align}
It is straightforward to check that these are indeed involutions, and that they commute since
\begin{align*}
\sigma(\tau(\alpha,\beta)) &= \sigma(\pi-\alpha, 2\pi-\beta) \\
&= (\pi-\alpha, 2\pi - (4(\pi-\alpha) + (2\pi-\beta))) \\
&\sim (\pi-\alpha, 2\pi - (2\pi - (4\alpha+\beta))) \\
&= \tau(\alpha, 2\pi-(4\alpha+\beta)) \\
&= \tau(\sigma(\alpha,\beta)).
\end{align*}

\begin{lemma} \label{lem:tau-involution}
Let $K$ be a knot in a homology sphere $Y$.  Then the pillowcase image $i^*X(Y,K)$ is fixed setwise by the involution $\tau$ of \eqref{eq:tau}.
\end{lemma}

\begin{proof}
Letting $E_K = Y \setminus N(K)$, we note that $H_1(E_K) \cong \Z$ is generated by the meridian $\mu$ of $K$, and that the longitude $\lambda$ is nullhomologous.  Given a point $(\alpha,\beta) \in i^*X(Y,K)$, corresponding to a representation $\rho: \pi_1(E_K) \to \SU(2)$ with
\begin{align*}
\rho(\mu) &= e^{i\alpha}, &
\rho(\lambda) &= e^{i\beta},
\end{align*}
we fix the central character $\chi: \pi_1(E_K) \twoheadrightarrow H_1(E_k) \to \{\pm1\}$ with $\chi(\mu) = -1$ and $\chi(\lambda) = 1$, and consider the representation
\[ \rho' = \chi \cdot \rho: \pi_1(E_K) \to \SU(2). \]
This satisfies $\rho'(\mu) = e^{i(\pi+\alpha)}$ and $\rho'(\lambda) = e^{i\beta}$, so its pillowcase image
\[ i^*([\rho']) = (\pi+\alpha,\beta) \sim (\pi-\alpha,2\pi-\beta) = \tau(\alpha,\beta) \]
belongs to $i^*X(K)$ as well.
\end{proof}

With Proposition~\ref{prop:order-4-gluing} in mind, we now introduce the following refinement of Proposition~\ref{prop:cyclic-curve}.

\begin{lemma} \label{lem:4-gamma}
Let $K \subset Y$ be a knot in a homology sphere, with irreducible, boundary-incompressible complement, and suppose that $Y_4(K)$ is a lens space.  Then there is a topologically embedded, closed curve
\[ \gamma \subset i^*X(Y,K) \]
that is homologically essential in the twice-punctured pillowcase
\[ P = X(T^2) \setminus \{(0,\pi), (\pi,\pi)\}, \]
and that contains the point $(\frac{\pi}{4},0)$.  This curve $\gamma$ does not contain any other point $(\alpha,\beta)$ with $4\alpha+\beta\equiv 0\pmod{\pi}$, except possibly for $(\frac{\pi}{2},0)$.
\end{lemma}

\begin{proof}
We let $\gamma_0$ be the curve provided by Proposition~\ref{prop:cyclic-curve}, noting that this curve avoids both $(0,0)$ and $(\pi,0)$.  Since $\gamma_0$ generates $H_1(P) \cong \Z$, it has intersection number $\pm1$ with the curve
\[ L_\pi = \{4\alpha+\beta \equiv \pi\!\!\!\pmod{2\pi}\} = \{ (\alpha, \pi-4\alpha) \mid 0<\alpha<\pi \} \]
that ends at the punctures $(0,\pi)$ and $(\pi,\pi)$, as shown in Figure~\ref{fig:4-avoiding}.

According to Lemma~\ref{lem:cyclic-surgery-su2}, the curve $\gamma_0$ can only meet $L_\pi$ at those points of $L_\pi$ where $\beta \equiv 0\pmod{2\pi}$, namely $(\frac{\pi}{4},0)$ and $(\frac{3\pi}{4},0)$.  If $\gamma_0$ passes through either of these points then there is a neighborhood of that point where $\gamma_0$ coincides with the edge $\{\beta\equiv 0\pmod{2\pi}\}$, by  Proposition~\ref{prop:zero-root-of-unity}, and so $\gamma_0$ meets $L_\pi$ transversely there.  This means that the points of
\[ \gamma_0 \cap L_\pi \subset \left\{ (\tfrac{\pi}{4},0), (\tfrac{3\pi}{4},0) \right\} \]
each contribute $\pm1$ to the intersection number $\gamma_0 \cdot L_\pi = \pm1$, and so exactly one of these must be a point of intersection.  If the point is $(\frac{\pi}{4},0)$, then we take $\gamma = \gamma_0$ and we are done.  If instead it is $(\frac{3\pi}{4},0)$, then we note that $i^*X(Y,K)$ is closed under the involution $\tau$ of \eqref{eq:tau}, by Lemma~\ref{lem:tau-involution}, so we take $\gamma = \tau(\gamma_0)$ instead.

As for other points $(\alpha,\beta) \in \gamma$ with $4\alpha+\beta\equiv 0\pmod{\pi}$, Lemma~\ref{lem:cyclic-surgery-su2} says that these must have $\beta\equiv 0 \pmod{2\pi}$, and then we have $\alpha = \frac{k\pi}{4}$ for some integer $k$, with $0\leq k \leq 4$ corresponding to $0\leq \alpha \leq \pi$.  We have already excluded $k=0,3,4$ and arranged that $(\frac{\pi}{4},0) \in \gamma$, so this leaves $(\frac{\pi}{2},0)$ (corresponding to $k=2$) as the only remaining possibility.
\end{proof}

\subsection{Intersecting curves of characters} \label{ssec:intersecting-curves}

The curves $\gamma$ produced by Lemma~\ref{lem:4-gamma} will be key to finding representations of closed 3-manifold groups.  We will repeatedly use the following two facts in our arguments:
\begin{itemize}
\item any pair of closed curves in the pillowcase has intersection number zero; 
\item and if $\gamma$ is an essential, simple closed curve in the twice-punctured pillowcase $P$, then it generates $H_1(P) \cong \Z$ and thus has intersection number $\pm1$ with any arc connecting the two punctures.
\end{itemize}
The second of these facts already played a role in the proof of Lemma~\ref{lem:4-gamma}.  We recall below that $\sigma$ denotes the involution \eqref{eq:sigma} of the pillowcase.

\begin{lemma} \label{lem:4-intersection}
Let $K_1 \subset Y_1$ and $K_2 \subset Y_2$ be knots in homology spheres.  Form a closed 3-manifold $Y$ by gluing together their exteriors $M_i = Y_i \setminus N(K_i)$ as in \eqref{eq:4-gluing}, so that
\begin{align*}
\mu_1 &\sim \mu_2, &
\lambda_1^{-1} &\sim \mu_2^4\lambda_2.
\end{align*}
If the intersection
\[ i^*X(Y_1,K_1) \cap \sigma\big(i^*X(Y_2,K_2)\big) \]
contains a point other than $(0,0)$, $(\frac{\pi}{2},0)$, and $(\pi,0)$, then there is a representation
\[ \pi_1(Y) \to \SU(2) \]
with non-abelian image.
\end{lemma}

\begin{proof}
Let $(\alpha,\beta)$ be a point of the intersection.  Then there are representations
\[ \rho_i: \pi_1(M_i) \to \SU(2) \]
with $\rho_1(\mu_1) = e^{i\alpha}$ and $\rho_1(\lambda_1) = e^{i\beta}$, and with 
\begin{align*}
\rho_1(\mu_1) &= e^{i\alpha} & \rho_2(\mu_2) &= e^{i\alpha_2} \\
\rho_1(\lambda_1) &= e^{i\beta}, & \rho_2(\lambda_2) &= e^{i\beta_2}
\end{align*}
such that $(\alpha,\beta) = \sigma(\alpha_2,\beta_2) = (\alpha_2,2\pi-(4\alpha_2+\beta_2))$.  In particular we have 
\begin{align*}
\rho_2(\mu_2) &= e^{i\alpha_2} = e^{i\alpha} = \rho_1(\mu_1), \\
\rho_2(\mu_2^4\lambda_2) &= e^{i(4\alpha_2+\beta_2)} = e^{i(2\pi-\beta)} = \rho_1(\lambda_1^{-1})
\end{align*}
and therefore $\rho_1$ and $\rho_2$ glue together to give a representation $\rho: \pi_1(Y) \to \SU(2)$ whose image contains the images of both $\rho_1$ and $\rho_2$.

Now if $\beta \not\equiv 0 \pmod{2\pi}$ then $\rho_1$ has non-abelian image, and likewise if $\beta_2 \not\equiv 0 \pmod{2\pi}$ then $\rho_2$ has non-abelian image.  Thus $\rho$ is non-abelian unless both $\beta$ and $\beta_2$ are multiples of $2\pi$.  Since $\sigma$ is an involution we have $(\alpha_2,\beta_2) = \sigma(\alpha,\beta) = (\alpha, 2\pi-(4\alpha+\beta))$, so $\beta \equiv \beta_2 \equiv 0 \pmod{2\pi}$ is equivalent to
\[ \beta \equiv 4\alpha+\beta \equiv 0 \pmod{2\pi} \]
and this corresponds to the three points $(0,0)$, $(\frac{\pi}{2},0)$, $(\pi,0)$ in the pillowcase.  Thus any intersection point away from these three gives rise to a non-abelian $\rho$, as desired.
\end{proof}

From now on we will repeatedly use the following hypotheses.

\begin{setup} \label{setup:toroidal}
Let $K_1 \subset Y_1$ and $K_2 \subset Y_2$ be knots in homology spheres, with the properties that each exterior $M_j = Y_j \setminus N(K_j)$ is irreducible and boundary-incompressible, and that each $4$-surgery
\[ (Y_j)_4(K_j), \qquad j=1,2 \]
is a lens space.  Form a closed 3-manifold $Y$ by gluing $M_1$ to $M_2$ by the map
\begin{align*}
\mu_1 &\sim \mu_2, \\
\lambda_1^{-1} &\sim \mu_2^4\lambda_2
\end{align*}
as in \eqref{eq:4-gluing}.  Finally, let $\gamma_j \subset i^*X(Y_j,K_j)$ be the embedded closed curves in the pillowcase provided by Lemma~\ref{lem:4-gamma}, each of which avoids the lines $\{ 4\alpha+\beta\in \pi\Z \}$ except at $(\frac{\pi}{4},0)$ and possibly $(\frac{\pi}{2},0)$.
\end{setup}

\begin{lemma} \label{lem:want-gamma-intersect}
Assume Setup~\ref{setup:toroidal}, and let
\[ c_j \subset i^*X(Y_j,K_j), \qquad j=1,2 \]
be closed, embedded curves that avoid the points $(0,0)$ and $(\pi,0)$, such as $\gamma_j$ or $\tau(\gamma_j)$.  If the intersection
\[ c_1 \cap \sigma(c_2) \]
is nonempty, then there is a non-abelian representation $\rho: \pi_1(Y) \to \SU(2)$.
\end{lemma}

\begin{proof}
By Lemma~\ref{lem:4-intersection} it suffices to show that $c_1$ and $\sigma(c_2)$ intersect in a point other than $(0,0)$, $(\frac{\pi}{2},0)$, or $(\pi,0)$.  By hypothesis they avoid the first and last of these, so we need only consider the case where $c_1$ and $\sigma(c_2)$ both pass through $(\frac{\pi}{2},0)$.  Since $(\frac{\pi}{2},0)$ is fixed by $\sigma$, this means that both of the $c_i$ contain $(\frac{\pi}{2},0)$.

According to  Proposition~\ref{prop:zero-root-of-unity}, there is a neighborhood of $(\frac{\pi}{2},0)$ in the pillowcase on which each $i^*X(Y_j,K_j)$ coincides with the line $\beta\equiv 0\pmod{2\pi}$.  This means that on a sufficiently small neighborhood $U$ of $(\frac{\pi}{2},0)$, we have
\begin{align*}
c_1 \cap U &= U \cap \{\beta\equiv 0\!\!\!\pmod{2\pi}\} \\
\sigma(c_2) \cap U &= U \cap \{4\alpha+\beta \equiv 0\!\!\!\pmod{2\pi}\},
\end{align*}
and so $c_1$ meets $\sigma(c_2)$ transversely at $(\frac{\pi}{2},0)$.  In particular, the point $(\frac{\pi}{2},0)$ contributes $\pm1$ to the intersection number
\[ c_1 \cdot \sigma(c_2) = 0, \]
so there must be at least one other point of intersection $(\alpha,\beta) \in c_1 \cap \sigma(c_2)$ and this provides the desired $\rho$.
\end{proof}

In what follows, we will repeatedly refer to the pair of arcs
\begin{equation} \label{eq:L-theta}
L_\theta = \{(\alpha,\beta) \mid 4\alpha+\beta \equiv \theta\!\!\!\pmod{2\pi} \} \qquad (\theta=0,\pi)
\end{equation}
in the pillowcase.

\begin{proposition} \label{prop:both-avoid-pi/2}
Assume Setup~\ref{setup:toroidal}, and suppose further that neither $\gamma_1$ nor $\gamma_2$ contains the point $(\frac{\pi}{2},0)$.  Then there is a non-abelian representation $\rho: \pi_1(Y) \to \SU(2)$.
\end{proposition}

\begin{proof}
We recall from Lemma~\ref{lem:tau-involution} that $\tau(\gamma_1) \subset i^*X(Y_1,K_1)$; moreover, the points $(0,0)$ and $(\pi,0)$ in the pillowcase are fixed by $\tau$, so if $\gamma_1$ avoids them both then so does $\tau(\gamma_1)$.  Thus by Lemma~\ref{lem:want-gamma-intersect}, it suffices to show that either $\gamma_1$ or $\tau(\gamma_1)$ intersects $\sigma(\gamma_2)$.  We will therefore suppose for the sake of a contradiction that
\[ \big( \gamma_1 \cup \tau(\gamma_1) \big) \cap \sigma(\gamma_2) = \emptyset. \]

We first record that according to Setup~\ref{setup:toroidal}, the curves $\gamma_1$ and $\gamma_2$ can only possibly meet the line $L_0$ at $(\frac{\pi}{2},0)$.  By assumption they both avoid this point, so in fact
\[ \gamma_1 \cap L_0 = \gamma_2 \cap L_0 = \emptyset. \]
Applying $\tau$ to $\gamma_1 \cap L_0$ and observing that $\tau(L_0) = L_0$ tells us that
\[ \tau(\gamma_1) \cap L_0 = \emptyset \]
as well, while if we apply $\sigma$ to $\gamma_2 \cap L_0$ then we see that
\[ \sigma(\gamma_2) \cap \{ \beta\equiv0 \!\!\! \pmod{2\pi} \} = \emptyset \]
since $\sigma$ sends $L_0$ to the line $\beta\equiv 0$.

By Proposition~\ref{prop:zero-root-of-unity}, the curve $\gamma_1$ intersects $L_\pi$ transversely at their sole point $(\frac{\pi}{4},0)$ of intersection, and similarly $\tau(\gamma_1)$ meets $L_\pi$ transversely at $(\frac{3\pi}{4},0)$ and nowhere else.  Thus $\gamma_1$ and $\tau(\gamma_1)$ separate $L_\pi$ into three segments, which we label
\begin{align*}
L_\pi^\ell &= \{(\alpha,\pi - 4\alpha) \mid 0 \leq \alpha < \tfrac{\pi}{4} \}, \\
L_\pi^m &= \{(\alpha,3\pi - 4\alpha) \mid \tfrac{\pi}{4} < \alpha < \tfrac{3\pi}{4} \}, \\
L_\pi^r &= \{(\alpha,5\pi - 4\alpha) \mid \tfrac{3\pi}{4} < \alpha \leq \pi \}.
\end{align*}
The union $\gamma_1 \cup  \tau(\gamma_1)$ splits the pillowcase into various path components, and we let
\[ X_\ell,\ X_m,\ X_r \subset X(T^2) \setminus \big( \gamma_1 \cup \tau(\gamma_1) \big) \]
denote the path components containing $L_\pi^\ell$, $L_\pi^m$, and $L_\pi^r$ respectively.  Since $L_\pi^\ell$ lies in a different component of $X(T^2) \setminus \gamma_1$ than $L_\pi^m$ and $L_\pi^r$, and similarly $L_\pi^r$ lies in a different component of $X(T^2) \setminus \tau(\gamma_1)$ than $L_\pi^\ell$ and $L_\pi^m$, it follows that the path components $X_\ell$, $X_m$, and $X_r$ are all distinct.

Next, we observe that the curves $\sigma(\gamma_2)$ and $L_0$ are by assumption disjoint from $\gamma_1 \cup \tau(\gamma_1)$, so each one lies entirely within some path component of the complement.  Since $\sigma$ restricts to an involution of the twice-punctured pillowcase, the curve $\sigma(\gamma_2)$ generates the homology of the latter just as $\gamma_2$ does; meanwhile $L_\pi$ is an arc with endpoints at the two punctures $(0,\pi)$ and $(\pi,\pi)$, so we must have
\[ \sigma(\gamma_2) \cap L_\pi \neq \emptyset. \]
At the same time $\sigma(\gamma_2)$ also contains the point $\sigma(\frac{\pi}{4},0) = (\frac{\pi}{4}, \pi) \in L_0$, so it contains a path with one endpoint on $L_0$ and the other endpoint on $L_\pi$.  The endpoint of this path on $L_\pi$ lies in one of the segments $L_\pi^\ell$, $L_\pi^m$, or $L_\pi^r$, since the remaining points $(\frac{\pi}{4},0)$ and $(\frac{3\pi}{4},0)$ of $L_\pi$ lie on $\gamma_1$ and $\tau(\gamma_1)$ respectively, so $L_0$ is in the same path component as that segment.  Thus exactly one of
\[ L_0 \subset X_\ell \quad\text{or}\quad L_0 \subset X_m \quad\text{or}\quad L_0 \subset X_r \]
must be true.

Since the set $\gamma_1 \cup \tau(\gamma_1)$ is $\tau$-invariant, the involution $\tau$ permutes the path components of their complement, and it moreover fixes the path component containing $L_0$ since $\tau(L_0) = L_0$.  We observe $\tau$ exchanges the points
\[ (0,\pi) \in L_\pi^\ell \subset X_\ell \qquad\text{and}\qquad (\pi,\pi) \in L_\pi^r \subset X_r, \]
but it fixes the point $(\frac{\pi}{2},\pi) \in L_\pi^m$, so $\tau$ exchanges $X_\ell$ and $X_r$ while fixing $X_m$ and therefore
\[ L_0 \subset X_m. \]

Finally, we build another path from $(0,\pi)$ to $(\pi,\pi)$ as the union
\[ L_\pi^\ell \cup \big( [\tfrac{\pi}{4},\tfrac{3\pi}{4}] \times \{0\} \big) \cup L_\pi^r. \]
Since $\sigma(\gamma_2)$ is homologically essential in the twice-punctured pillowcase, it must intersect this path somewhere.  But we saw that $\sigma(\gamma_2)$ is disjoint from the middle segment, since in fact it avoids the entire line $\{ \beta\equiv 0 \pmod{2\pi} \}$, so it follows that either
\[ \sigma(\gamma_2) \cap L_\pi^\ell \neq \emptyset \quad\text{or}\quad \sigma(\gamma_2) \cap L_\pi^r \neq \emptyset. \]
This means that $\sigma(\gamma_2)$ intersects at least one of the path components $X_\ell$ and $X_r$, and at the same time we have also seen that it contains the point
\[ (\tfrac{\pi}{4},\pi) \in L_0 \subset X_m. \]
But then the curve $\sigma(\gamma_2)$ contains a path from $X_m$ to either $X_\ell$ or $X_r$, contradicting the fact that each of these path components are distinct.  We conclude that $\sigma(\gamma_2)$ must intersect either $\gamma_1$ or $\tau(\gamma_1)$ after all, and this provides the desired representation $\rho$.
\end{proof}

\subsection{Instanton knot homology and $\SU(2)$ representations} \label{ssec:khi-representations}

In this subsection we suppose that $Y = M_1 \cup_{T^2} M_2$ is formed as in Setup~\ref{setup:toroidal}, and that $Y$ is $\SU(2)$-abelian.  We will use the curve of characters $\sigma(\gamma_2)$ for $K_2$ in the pillowcase, and specifically the fact that it mostly avoids the pillowcase image $i^*X(Y_1,K_1)$, to conclude that the \emph{instanton knot homology}
\[ \KHI(Y_1,K_1) \]
defined by Kronheimer and Mrowka \cite{km-excision} must be small. The following lower bound on the rank of $\KHI$ will then give us a contradiction, from which we can conclude that $Y$ must not be $\SU(2)$-abelian after all.

\begin{lemma} \label{lem:KHI-rank}
Let $K \subset Y$ be a knot in a homology sphere with irreducible, boundary-incompressible complement.  Then $\dim \KHI(Y,K) \geq 2$.
\end{lemma}

\begin{proof}
Kronheimer and Mrowka \cite{km-excision} define instanton knot homology as the sutured instanton homology
\[ \KHI(Y,K) = \SHI(Y(K)), \]
where the sutured manifold $Y(K)$ is the complement $M = Y\setminus N(K)$ with a pair of oppositely oriented meridional sutures.  Since $M = Y \setminus N(K)$ is irreducible, the sutured manifold $Y(K)$ is taut, and in particular its sutured instanton homology is nonzero \cite[Theorem~7.12]{km-excision}.

Now we know that $H_2(M)=0$ and that $\SHI(Y(K))$ is nonzero, so a theorem of Ghosh and Li \cite[Theorem~1.2]{gl-decomposition} tells us that
\[ \dim \SHI(Y(K)) < 2 \]
if and only if $Y(K)$ is a product sutured manifold, in which case
\[ Y(K) \cong (\Sigma \times [-1,1], \partial\Sigma \times \{0\}) \]
for some compact surface $\Sigma$ with boundary.  (This theorem is in turn a generalization of \cite[Theorem~7.18]{km-excision}, which required the additional hypothesis that $Y(K)$ is a homology product.)  But since the positive and negative regions $R_\pm \subset \partial M$ are annuli, this could only be possible if $\Sigma$ were an annulus, in which case $M \cong \Sigma \times [-1,1]$ would be a solid torus.  Since $\partial M$ is incompressible, we conclude that $\dim \KHI(Y,K) = \dim \SHI(Y(K))$ is at least 2 after all.
\end{proof}

Our goal will be to use $\sigma(\gamma_2)$ to construct a curve $\bar{c}' \subset X(T^2)$ and isotopy $h_t$ satisfying the hypotheses of the following theorem, and thus bound $\dim \KHI(Y_1,K_1)$ from above.

\begin{theorem}[{\cite[Theorem~4.8]{sz-pillowcase}}] \label{thm:khi-bound}
Let $K \subset Y$ be a knot in a homology 3-sphere, and suppose we have a smooth, simple closed curve $\bar{c}' \subset X(T^2)$ and an area-preserving isotopy
\[ h_t: X(T^2) \to X(T^2), \qquad 0 \leq t \leq 1 \]
that takes $\bar{c}'$ to the line $\{\alpha = \frac{\pi}{2}\}$ and fixes the four points $(0,0)$, $(0,\pi)$, $(\pi,0)$, and $(\pi,\pi)$.  Suppose moreover that
\begin{enumerate}
\item $\bar{c}'$ is disjoint from $i^*X^\irr(Y,K)$,
\item and $\bar{c}'$ intersects the line
\[ \{ \beta \equiv 0 \!\!\! \pmod{2\pi} \} \subset X(T^2) \]
transversely in $n$ points $(\alpha_1,0),\dots,(\alpha_n,0)$, with $\Delta_K(e^{2i\alpha_j}) \neq 0$ for all $j$.
\end{enumerate}
Then $\dim \KHI(K) \leq n$.
\end{theorem}

\begin{remark}
The statement of Theorem~\ref{thm:khi-bound} in \cite{sz-pillowcase} assumes that $Y \cong S^3$, but the proof applies verbatim when $Y$ is an arbitrary homology sphere.
\end{remark}

With these prerequisites at hand, we now devote the remainder of this subsection to the proof of the following proposition.

\begin{proposition} \label{prop:rho-from-KHI}
Assume Setup~\ref{setup:toroidal}, and suppose that $(\frac{\pi}{2},0)$ lies on at least one of the curves $\gamma_1$ and $\gamma_2$.  Then there is a representation $\rho: \pi_1(Y) \to \SU(2)$ with non-abelian image.
\end{proposition}

We begin with a special case of the proposition that will simplify our subsequent application of Theorem~\ref{thm:khi-bound} in the general case.

\begin{lemma} \label{lem:gamma-hits-lpi-and-pi/2}
Assume Setup~\ref{setup:toroidal}, and suppose in addition that $(\frac{\pi}{2},0) \in \gamma_2$.  If $\sigma(\gamma_2)$ intersects either of the segments
\begin{equation} \label{eq:L_pi-left-right}
\begin{aligned}
L_\pi^\ell &= \{ (\alpha,\pi-4\alpha) \mid 0 \leq \alpha < \tfrac{\pi}{4} \} \\
L_\pi^r &= \{ (\alpha, 5\pi-4\alpha) \mid \tfrac{3\pi}{4} < \alpha \leq \pi \}
\end{aligned}
\end{equation}
of the line $L_\pi = \{4\alpha+\beta \equiv \pi\pmod{2\pi} \}$, then there is a representation $\pi_1(Y) \to \SU(2)$ with non-abelian image.
\end{lemma}

\begin{proof}
By hypothesis the image $\sigma(\gamma_2)$ contains the point $\sigma(\frac{\pi}{2},0) = (\frac{\pi}{2},0)$.  Supposing for now that $\sigma(\gamma_2)$ intersects $L_\pi^\ell$, say at some point $p$, then the union
\[ L_\pi^\ell \cup \sigma(\gamma_2) \]
contains a path $\phi$ from $(0,\pi)$ to $(\frac{\pi}{2},0)$, realized by following $L_\pi^\ell$ from $(0,\pi)$ to $p$ and then traveling from $p$ to $(\frac{\pi}{2},0)$ along $\sigma(\gamma_2)$.  Now $\tau(\phi)$ is (up to reversing the direction of travel) a path from $\tau(\frac{\pi}{2},0) = (\frac{\pi}{2},0)$ to $\tau(0,\pi) = (\pi,\pi)$, and it is contained in
\[ \tau(L_\pi^\ell) \cup \tau\big(\sigma(\gamma_2)\big) = L_\pi^r \cup \sigma\big( \tau(\gamma_2) \big), \]
so we concatenate $\phi$ and $\tau(\phi)$ to get a path $\bar\phi$ inside
\begin{equation} \label{eq:lpi-union-gamma2}
L^\ell_\pi \cup \sigma(\gamma_2) \cup \sigma\big(\tau(\gamma_2)\big) \cup L^r_\pi
\end{equation}
from $(0,\pi)$ to $(\pi,\pi)$.

Since $\gamma_1$ is homologically essential in the twice-punctured pillowcase, the path $\bar\phi$ from one puncture to the other must intersect it somewhere.  Lemma~\ref{lem:cyclic-surgery-su2} guarantees that $\gamma_1$ is disjoint from both $L^\ell_\pi$ and $L^r_\pi$, so we must have either
\[ \gamma_1 \cap \sigma(\gamma_2) \neq \emptyset \quad\text{or}\quad \gamma_1 \cap \sigma\big(\tau(\gamma_2)\big) \neq \emptyset, \]
and in either case Lemma~\ref{lem:want-gamma-intersect} provides the desired representation $\pi_1(Y) \to \SU(2)$.

The case where $\sigma(\gamma_2)$ intersects $L^r_\pi$ is nearly identical, except that this leads to a path $\phi$ from $(\frac{\pi}{2},0)$ to $(\pi,\pi)$ inside $\sigma(\gamma_2) \cup L^r_\pi$.  In this case the union $\phi \cup \tau(\phi)$ still gives us a path from $(0,\pi)$ to $(\pi,\pi)$ inside \eqref{eq:lpi-union-gamma2} and we proceed exactly as before.
\end{proof}

We are looking for a non-abelian representation $\rho: \pi_1(Y) \to \SU(2)$, and if neither curve $\gamma_j$ passes through $(\frac{\pi}{2},0)$ then Proposition~\ref{prop:both-avoid-pi/2} provides such a $\rho$, so we will assume that at least one $\gamma_j$ does contain this point.  Since the roles of the two $(Y_j,K_j)$ are completely interchangeable, we will assume without loss of generality that
\[ \gamma_2 \cap L_0 = \big\{ (\tfrac{\pi}{2},0) \big\}. \]
With $L_\pi^\ell$ and $L_\pi^r$ defined as in \eqref{eq:L_pi-left-right}, we will also assume that
\begin{equation} \label{eq:gamma2-tau-lpi-ell}
\sigma(\gamma_2) \cap \left( L_\pi^\ell \cup L_\pi^r \right) = \emptyset,
\end{equation}
because otherwise Lemma~\ref{lem:gamma-hits-lpi-and-pi/2} provides a non-abelian representation $\pi_1(Y) \to \SU(2)$.

In addition to the above identification of $\gamma_2 \cap L_0$, we also know from Setup~\ref{setup:toroidal} that $\gamma_2 \cap L_\pi = \{(\frac{\pi}{4},0)\}$.  These claims about each $\gamma_2 \cap L_\theta$ are equivalent to
\begin{equation} \label{eq:sigma-gamma2-lines}
\begin{aligned}
\sigma(\gamma_2) \cap \{ (\alpha,\beta) \mid \beta\equiv 0\!\!\!\pmod{2\pi} \} &= \big\{ (\tfrac{\pi}{2},0) \big\}, \\
\sigma(\gamma_2) \cap \{ (\alpha,\beta) \mid \beta\equiv \pi\!\!\!\pmod{2\pi} \} &= \big\{ (\tfrac{\pi}{4},\pi) \big\}.
\end{aligned}
\end{equation}
Moreover, since $4$-surgery on $K_2 \subset Y_2$ is a lens space, Proposition~\ref{prop:zero-root-of-unity} says that $\gamma_2$ coincides with the line $\{\beta \equiv 0 \pmod{2\pi}\}$ on some neighborhood of $(\frac{\pi}{4},0)$ and of $(\frac{\pi}{2},0)$.  Applying $\sigma$, we see that there are likewise open neighborhoods of $(\tfrac{\pi}{4},\pi)$ and of $(\tfrac{\pi}{2},0)$ on which $\sigma(\gamma_2)$ coincides with the line $L_0 = \{ 4\alpha+\beta \equiv 0 \pmod{2\pi} \}$.  In particular, the intersections \eqref{eq:sigma-gamma2-lines} are both transverse.

The line segments $L^\ell_\pi$, $L^r_\pi$, and $\{ \beta \equiv 0, \pi \pmod{2\pi} \}$ collectively divide the pillowcase into a pair of closed disks of equal area: we write
\[ X(T^2) = D_{\mathrm{top}} \cup D_{\mathrm{bot}}, \]
where $D_{\mathrm{top}}$ and $D_{\mathrm{bot}}$ are the regions containing $(\frac{\pi}{2},\frac{3\pi}{2})$ and $(\frac{\pi}{2},\frac{\pi}{2})$ respectively, so that $\tau(D_{\mathrm{top}}) = D_{\mathrm{bot}}$.  In both sides of Figure~\ref{fig:construct-KHI-curve} we have shaded the region $D_{\mathrm{top}}$.  We note that 
\begin{equation} \label{eq:sigma-partial-dtop}
\sigma(\gamma_2) \cap \partial D_{\mathrm{top}} = \big\{ (\tfrac{\pi}{2},0),\ (\tfrac{\pi}{4},\pi) \big \}
\end{equation}
by \eqref{eq:gamma2-tau-lpi-ell} and \eqref{eq:sigma-gamma2-lines}, and that this intersection is transverse.

We now observe that by \eqref{eq:sigma-partial-dtop}, the intersection
\[ \sigma(\gamma_2) \cap D_{\mathrm{top}} \]
defines a path from $(\frac{\pi}{2},2\pi) = (\frac{\pi}{2},0)$ to $(\frac{\pi}{4},\pi)$ that avoids $\partial D_{\mathrm{top}}$ except at its endpoints.  This path must then cross the segment 
\[ L^m_\pi = \{ (\alpha,3\pi-4\alpha) \mid \tfrac{\pi}{4} < \alpha < \tfrac{3\pi}{4} \} \]
of $L_\pi$, because $L^m_\pi$ separates $(\frac{\pi}{2},2\pi)$ from $(\frac{\pi}{4},\pi)$ in the disk $D_{\mathrm{top}}$.  See the left side of Figure~\ref{fig:construct-KHI-curve}.
\begin{figure}
\begin{tikzpicture}[style=thick]
\begin{scope}
\fill[gray!15] (0,0) -- (0.75,0) -- (0,3) -- (3,3) -- (2.25,6) -- (0,6) -- cycle;
\begin{scope}[style=thin]
  \clip (0,0) rectangle (3,6);
  \foreach \i in {1,...,5} {
    \def\lcolor{\ifodd\i{red}\else{Green}\fi};
    \draw[\lcolor,densely dashed] (0,3*\i) -- ++(3,-12);
  }
  \node[Green,left, inner sep=1pt] at (1.1,1.6) {$L_0$};
\end{scope}
\begin{scope}
  \draw plot[mark=*,mark size = 0.5pt] coordinates {(0,0)(3,0)(3,6)(0,6)} -- cycle; 
  \begin{scope}[decoration={markings,mark=at position 0.6 with {\arrow[scale=1]{>}}}]
    \draw[postaction={decorate}] (3,6) -- (0,6);
    \draw[postaction={decorate}] (3,0) -- (0,0);
  \end{scope}
  \begin{scope}[decoration={markings,mark=at position 0.525 with {\arrow[scale=1]{>>}}}]
    \draw[postaction={decorate}] (0,0) -- (0,3);
    \draw[postaction={decorate}] (0,6) -- (0,3);
  \end{scope}
  \begin{scope}[decoration={markings,mark=at position 0.575 with {\arrow[scale=1]{>>>}}}]
    \draw[postaction={decorate}] (3,0) -- (3,3);
    \draw[postaction={decorate}] (3,6) -- (3,3);
  \end{scope}
  \draw[dotted] (0,3) -- (3,3);
  \draw[thin,|-|] (0,-0.4) node[below] {\small$0$} -- node[midway,inner sep=1pt,fill=white] {$\alpha$} ++(3,0) node[below] {\small$\vphantom{0}\pi$};
  \draw[thin,|-|] (0.75,-0.4) node[below] {\small$\frac{\vphantom{3}\pi}{4}$} -- node[midway,inner sep=1pt,fill=white] {$\alpha$} ++(1.5,0) node[below] {\small$\vphantom{0}\frac{3\pi}{4}$};
  \draw[thin,|-|] (-0.3,0) node[left] {\small$0$} -- node[midway,inner sep=1pt,fill=white] {$\beta$} ++(0,6) node[left] {\small$2\pi$};
\end{scope}
\begin{scope}[color=blue,style=ultra thick]
   \clip (0,0) rectangle (3,6);
   \draw (1.5,6) -- (1.6,5.6) -- ++(0.02,-0.02) ++(-0.02,0.02) plot [smooth] coordinates { (1.6,5.6) (2.4,4.8) (1.5,5.2) (1.125,4.5) (0.5,4.25) (2.5,3.75) (1.5,3.5) (0.55,3.8) (0.65,3.4) } -- (0.75,3) -- (0.85,2.6) plot[smooth] coordinates { (0.85,2.6) (2.4,1.5) (1.75,0.75) (1.4,1.25) (1.4,0.4) } -- (1.5,0);
\node[above right] at (1.6,1.85) {$\sigma(\gamma_2)$};
\end{scope}
\begin{scope}[color=red, style=ultra thick]
\clip (0,-0.1) rectangle (3,6.1);
\draw (0,3) -- (0.75,0) (2.25,6) -- (3,3);
\end{scope}
\draw[red,thin,-latex] (-0.5,0.6) node[left,inner sep=1pt]  {$L_\pi^\ell$} to[bend right=15] (0.5,0.5);
\draw[red,thin,-latex] (3.75,3.4) node[below,inner sep=1pt]  {$L_\pi^r$} to[bend right=30] (2.85,3.75);
\draw[red,thin,-latex] (-0.5,3.75) node[below left,inner sep=1pt] {$L^m_\pi$} to[bend left=15] (0.95,5);
\end{scope}
%%%
\begin{scope}[xshift=6cm]
\fill[gray!15] (0,0) -- (0.75,0) -- (0,3) -- (3,3) -- (2.25,6) -- (0,6) -- cycle;
\begin{scope}[style=thin]
  \clip (0,0) rectangle (3,6);
  \foreach \i in {1,...,5} {
    \def\lcolor{\ifodd\i{red}\else{Green}\fi};
    \draw[\lcolor,densely dashed] (0,3*\i) -- ++(3,-12);
  }
\end{scope}
\begin{scope}
  \draw plot[mark=*,mark size = 0.5pt] coordinates {(0,0)(3,0)(3,6)(0,6)} -- cycle; 
  \begin{scope}[decoration={markings,mark=at position 0.6 with {\arrow[scale=1]{>}}}]
    \draw[postaction={decorate}] (3,6) -- (0,6);
    \draw[postaction={decorate}] (3,0) -- (0,0);
  \end{scope}
  \begin{scope}[decoration={markings,mark=at position 0.525 with {\arrow[scale=1]{>>}}}]
    \draw[postaction={decorate}] (0,0) -- (0,3);
    \draw[postaction={decorate}] (0,6) -- (0,3);
  \end{scope}
  \begin{scope}[decoration={markings,mark=at position 0.575 with {\arrow[scale=1]{>>>}}}]
    \draw[postaction={decorate}] (3,0) -- (3,3);
    \draw[postaction={decorate}] (3,6) -- (3,3);
  \end{scope}
  \draw[dotted] (0,3) -- (3,3);
  \draw[thin,|-|] (0,-0.4) node[below] {\small$0$} -- node[midway,inner sep=1pt,fill=white] {$\alpha$} ++(3,0) node[below] {\small$\vphantom{0}\pi$};
  \draw[thin,|-|] (0.75,-0.4) node[below] {\small$\frac{\vphantom{3}\pi}{4}$} -- node[midway,inner sep=1pt,fill=white] {$\alpha$} ++(1.5,0) node[below] {\small$\vphantom{0}\frac{3\pi}{4}$};
  \draw[thin,|-|] (-0.3,0) node[left] {\small$0$} -- node[midway,inner sep=1pt,fill=white] {$\beta$} ++(0,6) node[left] {\small$2\pi$};
\end{scope}
\begin{scope}[color=blue,style=ultra thick]
   \draw[purple] (1.125,4.5) ++(0.02,0.02) -- ++(-0.02,-0.02) -- node[right,pos=0.7,inner sep=2pt] {$\ell_0$} (1.875,1.5) -- ++(-0.02,-0.02);
   \clip (2.25,6) -- (0.75,6) -- (1.5,3) -- (0,3) -- (0.75,0) -- (2.25,0) -- (1.5,3) -- (3,3) -- cycle;
   \foreach \x in {0,180} {
     \draw[rotate around={\x:(1.5,3)}] (1.5,6) -- (1.6,5.6) -- ++(0.02,-0.02) ++(-0.02,0.02) plot [smooth] coordinates { (1.6,5.6) (2.4,4.8) (1.5,5.2) (1.125,4.5) (0.5,4.25) };
   }
\end{scope}
\node[blue,above,inner sep=1pt] at (0.9,1.2) {$\small\tau(c_0)$};
\begin{scope}[color=red, style=thick]
\clip (0,-0.1) rectangle (3,6.1);
\draw (0,3) -- (0.75,0) (2.25,6) -- (3,3);
\end{scope}
\draw[fill=black] (1.125,4.5) circle (0.05) node[left,inner sep=2pt] {\footnotesize$p_0$};
\draw[fill=black] (1.875,1.5) circle (0.05) node[right,inner sep=2pt] {\footnotesize$\tau(p_0)$};
\node[blue,below,inner sep=1pt] at (2.3,4.8) {$c_0$};
\draw[red,thin,-latex] (-0.5,0.6) node[left,inner sep=1pt]  {$L_\pi^\ell$} to[bend right=15] (0.5,0.5);
\draw[red,thin,-latex] (3.75,3.4) node[below,inner sep=1pt]  {$L_\pi^r$} to[bend right=30] (2.85,3.75);
\draw[red,thin,-latex] (-0.5,3.75) node[below left,inner sep=1pt] {$L^m_\pi$} to[bend left=15] (0.95,5);
\end{scope}
\end{tikzpicture}
\caption{Left: the curve $\sigma(\gamma_2)$ avoids the line segments $L^\ell_\pi$ and $L^r_\pi$ and crosses $\beta \in \pi\Z$ only at $(\frac{\pi}{2},0)$ and $(\frac{\pi}{4},\pi)$, as promised in \eqref{eq:sigma-partial-dtop}. Right: we build a $\tau$-invariant, essential closed curve $\bar{c}$ in the pillowcase out of the arc $c_0$ from $(\frac{\pi}{2},2\pi)$ to $L_\pi^m$, a portion $\ell_0$ of $L^m_\pi$, and the image $\tau(c_0)$.  In both pictures we have shaded the disk $D_{\mathrm{top}}$ for clarity.}
\label{fig:construct-KHI-curve}
\end{figure}
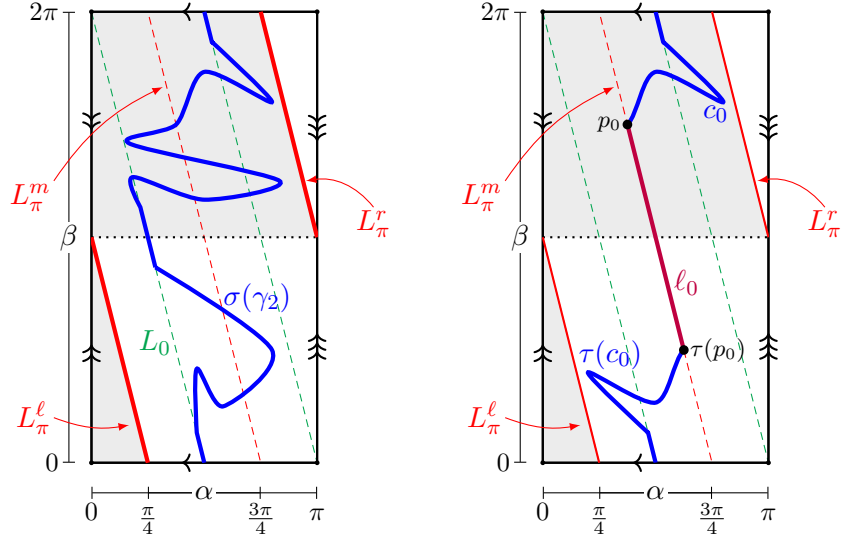

With this in mind, we let
\[ c_0 \subset \sigma(\gamma_2) \cap D_{\mathrm{top}} \]
be the portion of this path (including both endpoints) from $(\frac{\pi}{2},2\pi)$ to the first point
\[ p_0 = (\alpha_0,\beta_0) \in L^m_\pi \]
where this path meets $L^m_\pi$; we note that $\pi < \beta_0 < 2\pi$ and thus $\frac{\pi}{4} < \alpha_0 < \frac{\pi}{2}$.  We then take the line segment
\begin{equation} \label{eq:cbar-ell_0}
\ell_0 = \{ (\alpha, 3\pi-4\alpha) \mid \alpha_0 \leq \alpha \leq \pi-\alpha_0 \} \subset L^m_\pi
\end{equation}
from $p_0$ to $\tau(p_0) = (\pi-\alpha_0,2\pi-\beta_0)$, and define the simple closed curve
\[ \bar{c} = c_0 \cup \ell_0 \cup \tau(c_0) \]
in the pillowcase, as on the right side of Figure~\ref{fig:construct-KHI-curve}.  This is indeed a simple curve because by construction its three sections $c_0$, $\ell_0$, and $\tau(c_0)$ intersect each other only at their respective endpoints, which are among the points $(\frac{\pi}{2},0)$, $p_0$, and $\tau(p_0)$.  Moreover, since $\tau(\ell_0) = \ell_0$ it follows that $\tau(\bar{c}) = \bar{c}$.

\begin{lemma} \label{lem:cbar-neighborhood}
Suppose that there are no non-abelian representations $\pi_1(Y) \to \SU(2)$.  Then there is an open neighborhood $U$ of $\bar{c} \subset X(T^2)$ with the property that the intersection
\[ i^*X^\irr(Y_1,K_1) \cap U \]
is empty.
\end{lemma}

\begin{proof}
Proposition~\ref{prop:zero-root-of-unity} says that there is an open neighborhood $V \subset X(T^2)$ of the lines
\[ L_0 \cup L_\pi = \{ (\alpha,\beta) \mid 4\alpha+\beta \in \pi\Z \} \]
that is disjoint from $i^*X^\irr(Y_1,K_1)$.  Moreover, since $Y$ is $\SU(2)$-abelian, Lemma~\ref{lem:4-intersection} says that the entire pillowcase image $i^*X(Y_1,K_1)$ does not meet the subset
\[ c_0 \cup \tau(c_0) \subset \sigma(\gamma_2) \cup \tau\big(\sigma(\gamma_2)\big) = \sigma(\gamma_2) \cup \sigma\big(\tau(\gamma_2)\big) \]
of the pillowcase, except possibly at some of the points
\[ (0,0),\ (\tfrac{\pi}{2},0),\ (\pi,0) \in L_0 \subset V. \]
We let $V_0 \subset V$ be an open neighborhood of these three points, and then we have 
\[ i^*X^\irr(Y_1,K_1) \subset i^*X(Y_1,K_1) \setminus V_0. \]
The set on the right is closed and disjoint from the closed set $c_0 \cup \tau(c_0)$, so it is disjoint from an entire open neighborhood $W$ of $c_0 \cup \tau(c_0)$, and then
\[ i^*X^\irr(Y_1,K_1) \cap W = \emptyset. \]
We now let $U = V \cup W$, which contains all of $\bar{c}$ since $\ell_0 \subset V$ and $c_0 \cup \tau(c_0) \subset W$, and we observe that $U$ is disjoint from $i^*X^\irr(Y_1,K_1)$ since both $V$ and $W$ are.
\end{proof}

The curve $\bar{c}$ may not be smooth in general, but there is some small $\epsilon > 0$ such that it coincides with $L_0$ and with $L_\pi$ on $2\epsilon$-neighborhoods of $(\frac{\pi}{2},0)$ and of $(\frac{\pi}{2},\pi)$, respectively.  Given the neighborhood $U$ of Lemma~\ref{lem:cbar-neighborhood}, we can thus take a $C^0$-close approximation
\begin{equation} \label{eq:bar-c-prime}
\bar{c}' \subset U \subset X(T^2) \setminus i^*X^\irr(Y_1,K_1)
\end{equation}
of $\bar{c}$ such that
\begin{enumerate}
\item $\bar{c}'$ is smooth and $\tau$-invariant,
\item and $\bar{c}'$ coincides with $\bar{c}$ in $\epsilon$-neighborhoods of $(\frac{\pi}{2},0)$ and $(\frac{\pi}{2},\pi)$, i.e.,
\begin{align*}
\bar{c}' \cap D_{\epsilon}(\tfrac{\pi}{2},0) &= L_0 \cap D_\epsilon(\tfrac{\pi}{2},0), \\
\bar{c}' \cap D_{\epsilon}(\tfrac{\pi}{2},\pi) &= L_\pi \cap D_\epsilon(\tfrac{\pi}{2},\pi).
\end{align*}
In particular $\bar{c}'$ meets $\{\beta\equiv 0\pmod{2\pi}\}$ transversely at $(\frac{\pi}{2},0)$, just as $\bar{c}$ does.
\end{enumerate}
We achieve the $\tau$-invariance by first constructing the arc $\bar{c}' \cap D_{\mathrm{top}}$ and then using $\tau$ to extend it to $D_{\mathrm{bot}}$; if the initial arc is sufficiently $C^0$-close to $\bar{c}$ then both it and its image under $\tau$ will lie in $U$.

\begin{lemma} \label{lem:bar-c-isotopy}
There is an area-preserving isotopy of the pillowcase that fixes the corners $(0,0)$, $(0,\pi)$, $(\pi,0)$, and $(\pi,\pi)$, and that takes the curve $\bar{c}'$ of \eqref{eq:bar-c-prime} to the line $\{\alpha = \frac{\pi}{2}\}$.
\end{lemma}

\begin{proof}
By taking a smaller value of $\epsilon > 0$ as needed, we can assume that $\bar{c}'$ only intersects an $\epsilon$-neighborhood of $\partial D_{\mathrm{top}}$ near $(\frac{\pi}{2},0)$ and $(\frac{\pi}{2},\pi)$, where it coincides with $L_0$ and $L_\pi$ respectively.

We build a smooth curve $\bar{c}''$ isotopic to $\bar{c}'$ by taking a piecewise linear path in $D_{\mathrm{top}}$ from
\[ (\tfrac{\pi}{2}, 2\pi) \quad\text{to}\quad (\tfrac{\pi}{2}+\tfrac{\epsilon}{2}, 2\pi-2\epsilon) \quad\text{to}\quad (\tfrac{\pi}{2}-\tfrac{\epsilon}{2}, \tfrac{\pi}{2} + 2\epsilon) \quad\text{to}\quad (\tfrac{\pi}{2},\pi), \]
rounding corners in an $\frac{\epsilon}{2}$-neighborhood of the middle two vertices, and then using $\tau$ to extend this to $D_{\mathrm{bot}}$.  Then $\bar{c}''$ agrees with $\bar{c}'$ in an $\epsilon$-neighborhood of $\partial D_{\mathrm{top}}$, and we can arrange the corner-rounding process so that $\bar{c}''$ intersects every horizontal line of the form $\{\beta = \beta_0\}$, $\beta_0 \in \R/2\pi\Z$, in a single point.  See the middle of Figure~\ref{fig:KHI-curve-isotopy}.

We now choose a smooth isotopy of the disk $D_{\mathrm{top}}$ that fixes the above $\epsilon$-neighborhood of $\partial D_{\mathrm{top}}$ and takes the arc $\bar{c}' \cap D_{\mathrm{top}}$ to $\bar{c}'' \cap D_{\mathrm{top}}$.  We extend this $\tau$-equivariantly across $D_{\mathrm{bot}}$ to get an isotopy
\[ \phi_t: X(T^2) \to X(T^2), \qquad t\in [0,1] \]
that satisfies $\phi_0 = \mathrm{id}$ and $\phi_1(\bar{c}') = \bar{c}''$, and that fixes the corners $(0,0)$, $(0,\pi)$, $(\pi,0)$, and $(\pi,\pi)$.
\begin{figure}
\begin{tikzpicture}[style=thick,scale=0.9]
\begin{scope}
\fill[gray!15] (0,0) -- (0.75,0) -- (0,3) -- (3,3) -- (2.25,6) -- (0,6) -- cycle;
\begin{scope}[style=thin]
  \clip (0,0) rectangle (3,6);
  \foreach \i in {1,...,5} {
    \def\lcolor{\ifodd\i{red}\else{Green}\fi};
    \draw[\lcolor,densely dashed] (0,3*\i) -- ++(3,-12);
  }
\end{scope}
\begin{scope}
  \draw plot[mark=*,mark size = 0.5pt] coordinates {(0,0)(3,0)(3,6)(0,6)} -- cycle; 
  \begin{scope}[decoration={markings,mark=at position 0.6 with {\arrow[scale=1]{>}}}]
    \draw[postaction={decorate}] (3,6) -- (0,6);
    \draw[postaction={decorate}] (3,0) -- (0,0);
  \end{scope}
  \begin{scope}[decoration={markings,mark=at position 0.55 with {\arrow[scale=1]{>>}}}]
    \draw[postaction={decorate}] (0,0) -- (0,3);
    \draw[postaction={decorate}] (0,6) -- (0,3);
  \end{scope}
  \begin{scope}[decoration={markings,mark=at position 0.575 with {\arrow[scale=1]{>>>}}}]
    \draw[postaction={decorate}] (3,0) -- (3,3);
    \draw[postaction={decorate}] (3,6) -- (3,3);
  \end{scope}
  \draw[dotted] (0,3) -- (3,3);
  \draw[thin,|-|] (0,-0.4) node[below] {\small$0$} -- node[midway,inner sep=1pt,fill=white] {$\alpha$} ++(3,0) node[below] {\small$\vphantom{0}\pi$};
  \draw[thin,|-|] (0.75,-0.4) node[below] {\small$\frac{\vphantom{3}\pi}{4}$} -- node[midway,inner sep=1pt,fill=white] {$\alpha$} ++(1.5,0) node[below] {\small$\vphantom{0}\frac{3\pi}{4}$};
  \draw[thin,|-|] (-0.3,0) node[left] {\small$0$} -- node[midway,inner sep=1pt,fill=white] {$\beta$} ++(0,6) node[left] {\small$2\pi$};
\end{scope}
\begin{scope}[color=blue,style=ultra thick]
   \foreach \x in {0,180} {
     \draw[rotate around={\x:(1.5,3)}] plot [smooth] coordinates { (1.5,6) (1.575,5.7) (1.6,5.6) (2.4,4.8) (1.5,5.2) (1.125,4.5) (1.25,4) (1.5,3)};
   }
\end{scope}
\node[blue,right] at (1.375,3.5) {$\bar{c}'$};
\begin{scope}[color=red, style=thick]
\clip (0,-0.1) rectangle (3,6.1);
\draw (0,3) -- (0.75,0) (2.25,6) -- (3,3);
\end{scope}
\end{scope}
%%%
\draw[->] (3.25,3) -- node[above] {$\phi_t$} (4.25,3);
%%%
\begin{scope}[xshift=5cm]
\fill[gray!15] (0,0) -- (0.75,0) -- (0,3) -- (3,3) -- (2.25,6) -- (0,6) -- cycle;
\begin{scope}[style=thin]
  \clip (0,0) rectangle (3,6);
  \foreach \i in {1,...,5} {
    \def\lcolor{\ifodd\i{red}\else{Green}\fi};
    \draw[\lcolor,densely dashed] (0,3*\i) -- ++(3,-12);
  }
\end{scope}
\begin{scope}
  \draw plot[mark=*,mark size = 0.5pt] coordinates {(0,0)(3,0)(3,6)(0,6)} -- cycle; 
  \begin{scope}[decoration={markings,mark=at position 0.6 with {\arrow[scale=1]{>}}}]
    \draw[postaction={decorate}] (3,6) -- (0,6);
    \draw[postaction={decorate}] (3,0) -- (0,0);
  \end{scope}
  \begin{scope}[decoration={markings,mark=at position 0.55 with {\arrow[scale=1]{>>}}}]
    \draw[postaction={decorate}] (0,0) -- (0,3);
    \draw[postaction={decorate}] (0,6) -- (0,3);
  \end{scope}
  \begin{scope}[decoration={markings,mark=at position 0.575 with {\arrow[scale=1]{>>>}}}]
    \draw[postaction={decorate}] (3,0) -- (3,3);
    \draw[postaction={decorate}] (3,6) -- (3,3);
  \end{scope}
  \draw[dotted] (0,3) -- (3,3);
  \draw[thin,|-|] (0,-0.4) node[below] {\small$0$} -- node[midway,inner sep=1pt,fill=white] {$\alpha$} ++(3,0) node[below] {\small$\vphantom{0}\pi$};
  \draw[thin,|-|] (0.75,-0.4) node[below] {\small$\frac{\vphantom{3}\pi}{4}$} -- node[midway,inner sep=1pt,fill=white] {$\alpha$} ++(1.5,0) node[below] {\small$\vphantom{0}\frac{3\pi}{4}$};
  \draw[thin,|-|] (-0.3,0) node[left] {\small$0$} -- node[midway,inner sep=1pt,fill=white] {$\beta$} ++(0,6) node[left] {\small$2\pi$};
\end{scope}
\begin{scope}[color=blue,style=ultra thick]
   \foreach \x in {0,180} {
     \draw[rotate around={\x:(1.5,3)}] plot [smooth] coordinates { (1.5,6) (1.575,5.7) (1.575,5.5) (1.4575, 3.5) (1.4575,3.25) (1.5,3) };
   }
\end{scope}
\node[blue,right] at (1.45,3.5) {$\bar{c}''$};
\begin{scope}[color=red, style=thick]
\clip (0,-0.1) rectangle (3,6.1);
\draw (0,3) -- (0.75,0) (2.25,6) -- (3,3);
\end{scope}
\end{scope}
%%%
\draw[->] (8.25,3) -- node[above] {$\psi_t$} (9.25,3);
%%%
\begin{scope}[xshift=10cm]
\fill[gray!15] (0,0) -- (0.75,0) -- (0,3) -- (3,3) -- (2.25,6) -- (0,6) -- cycle;
\begin{scope}[style=thin]
  \clip (0,0) rectangle (3,6);
  \foreach \i in {1,...,5} {
    \def\lcolor{\ifodd\i{red}\else{Green}\fi};
    \draw[\lcolor,densely dashed] (0,3*\i) -- ++(3,-12);
  }
\end{scope}
\begin{scope}
  \draw plot[mark=*,mark size = 0.5pt] coordinates {(0,0)(3,0)(3,6)(0,6)} -- cycle; 
  \begin{scope}[decoration={markings,mark=at position 0.6 with {\arrow[scale=1]{>}}}]
    \draw[postaction={decorate}] (3,6) -- (0,6);
    \draw[postaction={decorate}] (3,0) -- (0,0);
  \end{scope}
  \begin{scope}[decoration={markings,mark=at position 0.55 with {\arrow[scale=1]{>>}}}]
    \draw[postaction={decorate}] (0,0) -- (0,3);
    \draw[postaction={decorate}] (0,6) -- (0,3);
  \end{scope}
  \begin{scope}[decoration={markings,mark=at position 0.575 with {\arrow[scale=1]{>>>}}}]
    \draw[postaction={decorate}] (3,0) -- (3,3);
    \draw[postaction={decorate}] (3,6) -- (3,3);
  \end{scope}
  \draw[dotted] (0,3) -- (3,3);
  \draw[thin,|-|] (0,-0.4) node[below] {\small$0$} -- node[midway,inner sep=1pt,fill=white] {$\alpha$} ++(3,0) node[below] {\small$\vphantom{0}\pi$};
  \draw[thin,|-|] (0.75,-0.4) node[below] {\small$\frac{\vphantom{3}\pi}{4}$} -- node[midway,inner sep=1pt,fill=white] {$\alpha$} ++(1.5,0) node[below] {\small$\vphantom{0}\frac{3\pi}{4}$};
  \draw[thin,|-|] (-0.3,0) node[left] {\small$0$} -- node[midway,inner sep=1pt,fill=white] {$\beta$} ++(0,6) node[left] {\small$2\pi$};
\end{scope}
\begin{scope}[color=blue,style=ultra thick]
   \draw (1.5,0) -- (1.5,6);
\end{scope}
\begin{scope}[color=red, style=thick]
\clip (0,-0.1) rectangle (3,6.1);
\draw (0,3) -- (0.75,0) (2.25,6) -- (3,3);
\end{scope}
\end{scope}
\end{tikzpicture}
\caption{A $\tau$-equivariant isotopy carrying $\bar{c}'$ to the straight line $\{\alpha=\frac{\pi}{2}\}$.  In each picture the disk $D_{\mathrm{top}}$ is shaded.}
\label{fig:KHI-curve-isotopy}
\end{figure}
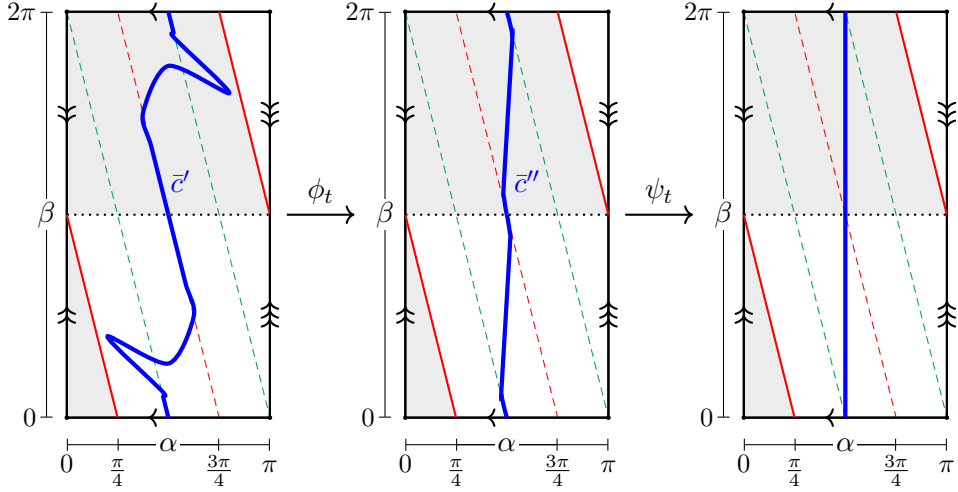
We then apply another $\tau$-invariant isotopy of the form
\[ \psi_t(\alpha,\beta) = (f_\beta(\alpha),\beta), \qquad t\in[0,1], \]
supported in the region
\[ \frac{\pi}{2}-\epsilon < \alpha < \frac{\pi}{2}+\epsilon \]
of the pillowcase, that takes $\phi_1(\bar{c}') = \bar{c}''$ to the line $\{\alpha = \frac{\pi}{2}\}$.  Each of these is illustrated in Figure~\ref{fig:KHI-curve-isotopy}.

The composition of $\phi$ and $\psi$ is an isotopy of the pillowcase that fixes each of the four corners, and that carries $\bar{c}'$ to $\{\alpha=\frac{\pi}{2}\}$ through $\tau$-invariant curves $\bar{c}'_t$.  These curves separate $X(T^2)$ into components that are exchanged by the isometry $\tau$ and must therefore have equal area.  We can therefore use \cite[Lemma~4.9]{sz-pillowcase} to promote $\phi \ast \psi$ to a smooth isotopy
\[ h_t: X(T^2) \to X(T^2), \qquad t\in[0,1], \]
fixing a neighborhood of the four corners for all $t$, such that each $h_t$ is a symplectomorphism satisfying $h_t(\bar{c}'_0) = \bar{c}'_t$.  Then $h_t$ is the desired isotopy from our original curve $\bar{c}'_0 = \bar{c}'$ to the line $\bar{c}'_1 = \{\alpha=\frac{\pi}{2}\}$.
\end{proof}

\begin{proof}[Proof of Proposition~\ref{prop:rho-from-KHI}]
As above, we assume without loss of generality that $(\frac{\pi}{2},0) \in \gamma_2$.  Assuming that the desired $\rho$ does not exist, Lemma~\ref{lem:gamma-hits-lpi-and-pi/2} says that the curve $\sigma(\gamma_2)$ is disjoint from the line segments $L^\ell_\pi$ and $L^r_\pi$ of \eqref{eq:L_pi-left-right}, so we have subsequently used $\sigma(\gamma_2)$ to construct
\begin{enumerate}
\item A smooth simple closed curve $\bar{c}'$ in the pillowcase (as in \eqref{eq:bar-c-prime}) such that
\begin{enumerate}
\item $\bar{c}'$ is disjoint from $i^*X^\irr(Y_1,K_1)$,
\item and the intersection 
\[ \bar{c}' \cap \{\beta\equiv 0 \!\!\! \pmod{2\pi} \} = \big\{ (\tfrac{\pi}{2},0) \big\} \]
is transverse.
\end{enumerate}
\item An area-preserving isotopy $h_t: X(T^2) \to X(T^2)$ that fixes the four corners of the pillowcase and sends $\bar{c}'$ to the line $\{\alpha = \frac{\pi}{2}\}$, by Lemma~\ref{lem:bar-c-isotopy}.
\end{enumerate}
Since $\Delta_{K_1}(e^{2i\cdot\pi/2}) = \Delta_{K_1}(-1)$ is nonzero, we can now apply Theorem~\ref{thm:khi-bound} to conclude that
\[ \dim \KHI(Y_1,K_1) \leq 1. \]
But this contradicts Lemma~\ref{lem:KHI-rank}, which asserts that
\[ \dim \KHI(Y_1,K_1) \geq 2. \]
The representation $\rho$ must therefore exist after all.
\end{proof}

We can finally prove the main theorem of this section.

\begin{proof}[Proof of Theorem~\ref{thm:nullhomologous-case}]
Proposition~\ref{prop:order-4-gluing} says that we can write $Y = M_1 \cup_{T^2} M_2$, where each $M_j$ is the exterior of a knot $K_j$ in a homology sphere $Y_j$ and the gluing maps satisfy
\begin{align*}
\mu_1 &\sim \mu_2, &
\lambda_1^{-1} &\sim \mu_2^4\lambda_2
\end{align*}
as in \eqref{eq:4-gluing}.  Then $(Y_1)_4(K_1) \cong M_1(\lambda_2)$ and $(Y_2)_4(K_2) \cong M_2(\lambda_1)$ are both lens spaces of order $4$, so Lemma~\ref{lem:4-gamma} provides us with simple closed curves $\gamma_j \subset i^*X(Y_j,K_j)$ for $j=1,2$, and all of the above is exactly as described in Setup~\ref{setup:toroidal}.

Now if neither $\gamma_1$ nor $\gamma_2$ contains the point $(\frac{\pi}{2},0)$, then $\rho$ is provided by Proposition~\ref{prop:both-avoid-pi/2}.  Otherwise at least one of the $\gamma_j$ contains $(\frac{\pi}{2},0)$, and then Proposition~\ref{prop:rho-from-KHI} tells us that $\rho$ must exist.
\end{proof}

Since Theorem~\ref{thm:nullhomologous-case} has now been proved, this completes the proof of Theorem~\ref{thm:main-Z/4}. \hfill\qed

\bibliographystyle{myalpha}
\bibliography{References}

\end{document}